\documentclass[final,leqno]{siamltex}

\usepackage{graphicx}
\usepackage{amsmath,amsfonts}
\usepackage[framemethod=tikz]{mdframed}
\usepackage{showkeys}

\def\beq{\begin{equation}}
\def\eeq{\end{equation}}

\def\E{\mathbb{E}}

\def\R{\mathbb{R}}

\newtheorem{assumption}{Assumption}
\newtheorem{example}{Example}
\newcommand{\half}{\mbox{${1 \over 2}$}}
\def\ba{\begin{array}}
\def\ea{\end{array}}
\def\beann{\begin{eqnarray*}}
\def\eeann{\end{eqnarray*}}
\def\bea{\begin{eqnarray}}
\def\eea{\end{eqnarray}}

\def\BT{\begin{theorem}}
\def\ET{\end{theorem}}
\def\BL{\begin{lemma}}
\def\EL{\end{lemma}}
\def\BC{\begin{corollary}}
\def\EC{\end{corollary}}
\def\BE{\begin{example}}
\def\EE{\end{example}}
\def\BD{\begin{definition}}
\def\ED{\end{definition}}
\def\BR{\begin{remark}}
\def\ER{\end{remark}}
\def\BAS{\begin{assumption}}
\def\EAS{\end{assumption}}
\def\BI{\begin{itemize}}
\def\EI{\end{itemize}}

\def\BMP{\begin{minipage}{9.5cm}}
\def\EMP{\end{minipage}}
\def\MPT{\begin{minipage}{11.5cm}}
\def\EPT{\end{minipage}}

\def\la{\langle}
\def\ra{\rangle}

\def\theequation{\thesection.\arabic{equation}}

\newtheorem{remark}[theorem]{Remark}

\slugger{siopt}{}{30}{4}{2750-2779}

\title{Tensor Methods for Minimizing Convex Functions with H\"{o}lder Continuous Higher-Order Derivatives}

\author{G.N. Grapiglia\thanks{Departamento de Matem\'atica, Universidade Federal do Paran\'a, Centro Polit\'ecnico, Cx. postal 19.081, 81531-980, Curitiba, Paran\'a, Brazil (grapiglia@ufpr.br). This author was supported by the National Council for Scientific and Technological Development - Brazil (grants 401288/2014-5 and 406269/2016-5) and by the European Research Council (ERC) under the European Union’s Horizon 2020 research and innovation programme (grant agreement No. 788368).}
        \and Yu. Nesterov\thanks{Center for Operations Research and Econometrics (CORE), Catholic University of Louvain
(UCL), 34 voie du Roman Pays, 1348 Louvain-la-Neuve,
Belgium (Yurii.Nesterov@uclouvain.be). This project has received funding from the European Research Council (ERC) under the European Union’s Horizon 2020 research and innovation programme (grant agreement No. 788368).}}

\date{November 28, 2019}

\begin{document}

\maketitle

\begin{abstract}
In this paper we study $p$-order methods for unconstrained minimization of convex functions that are $p$-times differentiable ($p\geq 2$) with $\nu$-H\"{o}lder continuous $p$th derivatives. We propose tensor schemes with and without acceleration. For the schemes without acceleration, we establish iteration complexity bounds of $\mathcal{O}\left(\epsilon^{-1/(p+\nu-1)}\right)$ for reducing the functional residual below a given $\epsilon\in (0,1)$. Assuming that $\nu$ is known, we obtain an improved complexity bound of $\mathcal{O}\left(\epsilon^{-1/(p+\nu)}\right)$ for the corresponding accelerated scheme. For the case in which $\nu$ is unknown, we present a universal accelerated tensor scheme with iteration complexity of $\mathcal{O}\left(\epsilon^{-p/[(p+1)(p+\nu-1)]}\right)$. A lower complexity bound of $\mathcal{O}\left(\epsilon^{-2/[3(p+\nu)-2]}\right)$ is also obtained for this problem class.

\end{abstract}



\pagestyle{myheadings}
\thispagestyle{plain}
\markboth{Tensor Methods for Minimizing Convex Functions}{G.N. Grapiglia and Yu. Nesterov}

\section{Introduction}
\setcounter{equation}{0}

\subsection{Motivation}
In \cite{NP}, it was shown that a suitable cubic
regularization of the Newton method (CNM) takes at most
$\mathcal{O}(\epsilon^{-1/2})$ iterations to reduce the
functional residual below a given precision $\epsilon>0$ when the objective is a twice-differentiable convex
function with a Lipschitz continuous Hessian. A better
complexity bound of $\mathcal{O}(\epsilon^{-1/3})$ was
shown in \cite{NES3} for an accelerated version of CNM.
Auxiliary problems in these methods consist in the
minimization of a third-order regularization of the
second-order Taylor approximation of the objective
function around the current i\-te\-ra\-te. A natural
generalization is to consider auxiliary problems in which
one minimizes a $(p+1)$-order regularization of the $p$th-order Taylor approximation of the objective function,
resulting in tensor methods. Unconstrained
optimization by tensor methods is not a new subject (see,
for example, \cite{SCH,BOU}). In the context of convex
optimization, accelerated tensor methods (as
described above) were first considered by Baes
\cite{BAES}. However, the author did not realize that
under a proper choice of the regularization coefficient
the auxiliary problems become convex. This important
observation was done in a recent paper \cite{NES6}, where
tensor methods with and wi\-thout
acce\-le\-ra\-tion were proposed for unconstrained
minimization of $p$-times differentiable convex functions
with Lipschitz continuous $p$th derivatives. An
iteration complexity bound of
$\mathcal{O}(\epsilon^{-1/p})$ was proved for the method
without acceleration, while an improved bound of
$\mathcal{O}(\epsilon^{-1/(p+1)})$ was proved for the
accelerated tensor method.

In the present paper, we study tensor methods (with and without acceleration) that can handle convex functions with $\nu$-H\"{o}lder continuous $p$th derivatives ($p\geq 2$) and allow the inexact solution of auxiliary problems (in the sense of [4]). Specifically, our contribution is threefold:
\begin{enumerate}
\item For the schemes without acceleration, we establish iteration complexity bounds of $\mathcal{O}\left(\epsilon^{-1/(p+\nu-1)}\right)$ for reducing the functional residual below a given $\epsilon\in (0,1)$.
\item[]
\item Assuming that $\nu$ is known, we obtain an improved complexity bound of $\mathcal{O}\left(\epsilon^{-1/(p+\nu)}\right)$ for the corresponding accelerated scheme. For the case in which $\nu$ is unknown, we present a \textit{universal} accelerated tensor scheme with iteration complexity of $\mathcal{O}\left(\epsilon^{-p/[(p+1)(p+\nu-1)]}\right)$.
\item[]
\item A lower complexity bound of $\mathcal{O}\left(\epsilon^{-2/[3(p+\nu)-2]}\right)$ is also obtained, from which we conclude that our accelerated nonuniversal scheme is nearly optimal.
\item[]
\end{enumerate}

The methods and results presented here extend in a significant way the contributions in \cite{BAES,GN,GN2,NES6}. Indeed, \cite{GN,GN2} deal only with second-order schemes ($p=2$) which require the exact solution of the auxiliary problems. On the other hand, the $p$-order methods proposed in \cite{BAES,NES6} are restricted to the Lipschitz case ($\nu=1$), assuming that the Lipschitz constant is known and that the auxiliary problems are solved exactly. We believe that the development of $p$-order methods with affordable trial steps and automatic adjustment to the objective's function class (universality) constitutes a fundamental step towards implementable high-order methods for convex optimization. 

\subsection{Contents}
The paper is organized as follows. In Section 2, we define
our problem. In Section 3, we present tensor
methods without acceleration and establish their
convergence properties. In Section 4, we present complexity
results for accelerated schemes. Finally, in Section 5 we
obtain lower complexity bounds for tensor methods under
the H\"{o}lder condition. All necessary auxiliary results
are included in Appendix A.

\subsection{Notations and Generalities}

In what follows, we denote by $\E$ a finite-dimensional
real vector space, and by $\E^*$ its {\em dual} space,
composed by linear functionals on $\E$. The value of
function $s \in \E^*$ at point $x \in \E$ is denoted by
$\la s, x \ra$. Given a self-adjoint positive definite
operator $B:\E \to \E^*$ (notation $B \succ 0$), we can
endow these spaces with conjugate Euclidean norms:
$$
\ba{rcl}
\| x \| & = & \la B x, x \ra^{1/2},\; x\in \E, \quad \| s
\|_* \; = \; \la s, B^{-1} s \ra^{1/2}, \; s \in \E^*.
\ea
$$
For a smooth function $f:\text{dom}\,f\to \R$ with convex and
open domain $\text{dom}\,f\subset\E$, denote by $\nabla f(x)$
its gradient and by $\nabla^{2}f(x)$ its Hessian evaluated at
point $x\in\text{dom}\,f$. Note that $\nabla f(x)\in\E^{*}$ and
$\nabla^{2}f(x)h\in\E^{*}$ for $x\in\text{dom}\,f$ and $h\in\E$.

For any integer $p\geq 1$, denote by
\begin{equation*}
D^{p}f(x)[h_{1},\ldots,h_{p}]
\end{equation*}
the directional derivative of function $f$ at $x$ along directions
$h_{i}\in\E$, $i=1,\ldots,p$. In particular, for any $x\in\text{dom}\,f$ and $h_{1},h_{2}\in\E$ we have
\begin{equation*}
Df(x)[h_{1}]=\langle\nabla f(x),h_{1}\rangle\quad\text{and}\quad D^{2}f(x)[h_{1},h_{2}]=\langle\nabla^{2}f(x)h_{1},h_{2}\rangle.
\end{equation*}
For $h_{1}=\ldots=h_{p}=h\in\E$, we use notation
$D^{p}f(x)[h]^{p}$. Then the $p$th-order Taylor
approximation of function $f$ at $x\in\text{dom}\,f$ can
be written as follows:
\begin{equation}
f(x+h)=\Phi_{x,p}(x+h)+o(\|h\|^{p}),\,\,x+h\in\text{dom}\,f,
\label{eq:1.1}
\end{equation}
where
\begin{equation}
\Phi_{x,p}(y)\equiv f(x)+\sum_{i=1}^{p}\dfrac{1}{i!}D^{i}f(x)[y-x]^{i},\,\,y\in\E.
\label{eq:1.2}
\end{equation}
Note that $D^{p}f(x)[\,.\,]$ is a symmetric $p$-linear
form. Its norm is defined by
\begin{equation*}
\|D^{p}f(x)\|=\max_{h_{1},\ldots,h_{p}}\left\{\left|D^{p}f(x)[h_{1},\ldots,h_{p}]\right|\,:\,\|h_{i}\|\leq 1,\,i=1,\ldots,p\right\}.
\end{equation*}
In fact, it can be shown that (see, e.g., \cite{BAN})
\begin{equation*}
\|D^{p}f(x)\|=\max_{h}\left\{\left|D^{p}f(x)[h]^{p}\right|\,:\,\|h\|\leq 1\right\}.
\end{equation*}
Similarly, since
$D^{p}f(x)[.\,,\ldots,\,.]-D^{p}f(y)[.,\ldots,.]$ is also
a symmetric $p$-linear form for fixed
$x,y\in\text{dom}\,f$, we can define
 \begin{equation*}
\|D^{p}f(x)-D^{p}f(y)\|=\max_{h}\left\{\left|D^{p}f(x)[h]^{p}-D^{p}f(y)[h]^{p}\right|\,:\,\|h\|\leq 1\right\}.
\end{equation*}

\section{Problem Statement}

In this paper we consider methods for solving the following minimization problem
\begin{equation}
\min_{x\in\E}\,f(x),
\label{eq:2.1}
\end{equation}
where $f:\E\to\R$ is a convex $p$-times differentiable
function ($p\geq 2$). We assume that there exists at least one optimal
solution $x^{*}\in\E$ for problem (\ref{eq:2.1}). Let us
characterize the level of smoothness of the objective $f$
by the system of H\"{o}lder constants
\begin{equation}
H_{f,p}(\nu)\equiv\sup_{x,y\in\E}\left\{\dfrac{\|D^{p}f(x)-D^{p}f(y)\|}{\|x-y\|^{\nu}}\,:\,x\neq y\right\},\,\,0\leq \nu\leq 1.
\label{eq:2.3}
\end{equation}
Then, from (\ref{eq:2.3}) and from the integral form of
the mean-value theorem, it follows that
\begin{equation}
|f(y)-\Phi_{x,p}(y)|\leq\dfrac{H_{f,p}(\nu)}{p!}\|y-x\|^{p+\nu},
\label{eq:2.4}
\end{equation}
\begin{equation}
\|\nabla f(y)-\nabla\Phi_{x,p}(y)\|_{*}\leq\dfrac{H_{f,p}(\nu)}{(p-1)!}\|y-x\|^{p+\nu-1},
\label{eq:2.5}
\end{equation}
for all $x,y\in\E$. Given $x\in\E$, if $H_{f,p}(\nu)<+\infty$ and $H\geq H_{f,p}(\nu)$, by (\ref{eq:2.4}) we have
\begin{equation}
f(y)\leq \Phi_{x,p}(y)+\dfrac{H}{p!}\|y-x\|^{p+\nu},\,\,y\in\E.
\label{eq:2.6}
\end{equation}
This property motivates the use of the following class of models of $f$ around $x\in\E$:
\begin{equation}
\Omega_{x,p,H}^{(\alpha)}(y)=\Phi_{x,p}(y)+\dfrac{H}{p!}\|y-x\|^{p+\alpha},\,\,\alpha\in [0,1].
\label{eq:2.7}
\end{equation}
In particular, as long as $H\geq H_{f,p}(\nu)$, by
(\ref{eq:2.6}) we have
\begin{equation}
f(y)\leq \Omega_{x,p,H}^{(\nu)}(y),\quad y\in\E.
\label{eq:2.8}
\end{equation}

\section{Tensor schemes without acceleration}

If we assume that $H_{f,p}(\nu)<+\infty$ for some $\nu\in
[0,1]$, there are two possible situations: either $\nu$ is
known, or $\nu$ is unknown. We cover both cases in a
single framework by introducing parameter
\begin{equation}
\alpha=\left\{\begin{array}{ll} \nu,&\text{if}\,\,\nu\,\,\text{is known},\\
                                1, &\text{if}\,\,\nu\,\,\text{is unknown}.
              \end{array}
       \right.
\label{eq:3.1}
\end{equation}
Let $\epsilon\in (0,1)$ be the target precision. At the
beginning of the $t$th iteration one has an estimate
$x_{t}$ for the solution of (\ref{eq:2.1}) and a scaling
coefficient $M_{t}$. A trial point $x_{t}^{+}$ is computed
as an approximate solution to the auxiliary
problem
\begin{equation}
\min_{y\in\E}\,\Omega_{x_{t},p,M_{t}}^{(\alpha)}(y),
\label{eq:3.2}
\end{equation}
with $\alpha$ given by (\ref{eq:3.1}). Similarly to \cite{Birgin}, the trial point $x^{+}_{t}$ must satisfy the following conditions:
\begin{equation}
\Omega_{x_{t},p,M_{t}}^{(\alpha)}(x_{t}^{+})\leq f(x_{t})\quad\text{and}\quad \|\nabla\Omega_{x_{t},p,M_{t}}^{(\alpha)}(x_{t}^{+})\|_{*}\leq\theta\|x_{t}^{+}-x_{t}\|^{p+\alpha-1},
\label{eq:3.3}
\end{equation}
where $\theta\geq 0$ is a user-defined parameter. When (\ref{eq:3.2}) is not convex, then $x_{t}^{+}$ is not necessarily an approximation of its global solution. If the descent condition
\begin{equation}
f(x_{t})-f(x_{t}^{+})\geq\dfrac{1}{8(p+1)!M_{t}^{\frac{1}{p+\alpha-1}}}\|\nabla f(x_{t}^{+})\|_{*}^{\frac{p+\alpha}{p+\alpha-1}},
\label{eq:3.5}
\end{equation}
holds, then $x_{t}^{+}$ is accepted and we define
$x_{t+1}=x_{t}^{+}$. Otherwise, constant $M_{t}$ is
increased until the corresponding trial point $x_{t}^{+}$
is accepted. We will see that this process is well defined
in the sense that there exists $M_{\nu}>0$ such that
$M_{t}\leq M_{\nu}$ for all $t$. This general scheme can
be summarized in the following way.
\\[0.1cm]
\begin{mdframed}
\noindent\textbf{Algorithm 1. Tensor Method}
\\[0.2cm]
\noindent\textbf{Step 0.} Choose $x_{0}\in\E$ and $\theta\geq 0$.  Set $\alpha$ by (\ref{eq:3.1}) and $t:=0$.\\
\noindent\textbf{Step 1.} Find $0<M_{t}\leq M_{\nu}$ such that (\ref{eq:3.5}) holds for an approximate solution $x_{t}^{+}$ to (\ref{eq:3.2}) satisfying conditions (\ref{eq:3.3}).\\
\noindent\textbf{Step 2.} Set $x_{t+1}=x_{t}^{+}$.\\
\noindent\textbf{Step 3.} Set $t:=t+1$ and go back to Step 1.
\end{mdframed}

\begin{remark}
Regarding the approximate solution of the auxiliary problems, it is easy to see that $x_{t}^{+}$ satisfying (3.3) can be computed by any monotone optimization scheme that drives the gradient of the objective to zero. In \cite{GN3}, we investigated the possible use of gradient methods with Bregman distance for the case $p=3$ and $\alpha=\nu$. Specifically, under assumptions H1 and H2 below, we showed that if $\left\{y_{k}\right\}_{k=0}^{T}$ is a sequence generated by these methods applied to $\min_{y\in\mathbb{E}}\Omega_{x_{t},3,M_{t}}^{(\nu)}(y)$ with 
\begin{equation*}
\|\nabla f(y_{k})\|_{*}>\epsilon\quad\text{and}\quad\|\nabla\Omega_{x_{t},3,M_{t}}^{(\nu)}(y_{k})\|_{*}>\theta\|y_{k}-x_{0}\|^{2+\nu},\quad\text{for}\,\, k=0,\ldots,T,
\end{equation*}
then
\begin{equation*}
T\leq\left\{\begin{array}{ll} \mathcal{O}\left(\epsilon^{-(3+\nu)}\right),&\text{if}\,\,M_{t}<[6/(3+\nu)]H_{f,3}(\nu)\,\,\text{and}\,\,\nu\neq 0,\\
                              \mathcal{O}\left(\epsilon^{-\frac{(3+\nu)}{2}}\right),&\text{if}\,\,M_{t}=[6/(3+\nu)]H_{f,3}(\nu)\,\,\text{and}\,\,\nu\neq 0,\\
                              \mathcal{O}\left(\log(\epsilon^{-1})\right),&\text{if}\,\,M_{t}>[6/(3+\nu)]H_{f,3}(\nu).\\
            \end{array}
      \right.
\end{equation*}
For more details, see Theorems 3.8 and 3.10 in \cite{GN3}.
\end{remark}

\noindent To analyze the convergence of Algorithm 1, we introduce the
following assumptions:
\begin{itemize}
\item[\textbf{H1}] $H_{f,p}(\nu)<+\infty$ for some $\nu\in [0,1]$.
\item[\textbf{H2}] The level sets of $f$ are bounded, that is, $\max_{x\in \mathcal{L}(x_{0})}\|x-x^{*}\|\leq D_{0}\in (1,+\infty)$ for $\mathcal{L}(x_{0})\equiv\left\{x\in\E\,:\,f(x)\leq f(x_{0})\right\}$, with $x_{0}$ being the starting point.
\end{itemize}
The next theorem establishes global convergence rate for
Algorithm 1.

\begin{theorem}
\label{thm:3.1}
Suppose that H1 and H2 are true and let $\left\{x_{t}\right\}_{t=0}^{T}$ be a sequence generated by Algorithm 1. Denote by $m$ the first iteration number such that
\begin{equation*}
f(x_{m})-f(x^{*})\leq 2[8(p+1)!]^{p+\alpha-1}M_{\nu}D_{0}^{p+\alpha},
\end{equation*}
and assume that $m < T$. Then
\begin{equation}
m\leq\dfrac{1}{\ln\left(\frac{p+\alpha}{p+\alpha-1}\right)}\ln\max\left\{1,\log_{2}\dfrac{f(x_{0})-f(x^{*})}{[8(p+1)!]^{p+\alpha-1}M_{\nu}D_{0}^{p+\alpha}}\right\}
\label{eq:3.13}
\end{equation}
and, for all $k$, $m<k\leq T$, we have
\begin{equation}
f(x_{k})-f(x^{*})\leq\dfrac{[24p(p+1)!]^{p+\alpha-1}M_{\nu}D_{0}^{p+\alpha}}{(k-m)^{p+\alpha-1}}.
\label{eq:3.14}
\end{equation}
\end{theorem}

\begin{proof}
By Step 1 in Algorithm 1, we have
\begin{equation}
M_{k}\leq M_{\nu},\,\,k=0,\ldots,T-1.
\label{eq:3.16}
\end{equation}
Thus, in view of (\ref{eq:3.5}), (\ref{eq:3.16}), and H2, for $k=0,\ldots,T-1$ we have
\begin{eqnarray}
f(x_{k})-f(x_{k+1})&\geq &\dfrac{1}{8(p+1)!}\left[\dfrac{1}{M_{k}}\right]^{\frac{1}{p+\alpha-1}}\|\nabla f(x_{k+1})\|_{*}^{\frac{p+\alpha}{p+\alpha-1}}\nonumber\\
                  &\geq &\dfrac{1}{8(p+1)!}\left[\dfrac{1}{M_{\nu}}\right]^{\frac{1}{p+\alpha-1}}\|\nabla f(x_{k+1})\|_{*}^{\frac{p+\alpha}{p+\alpha-1}}\label{eq:3.17}\\
                  & \geq &\dfrac{1}{8(p+1)!}\left[\dfrac{1}{M_{\nu}D_{0}^{p+\alpha}}\right]^{\frac{1}{p+\alpha-1}}\left(\|\nabla f(x_{k+1})\|_{*}\|x_{k+1}-x^{*}\|\right)^{\frac{p+\alpha}{p+\alpha-1}}\nonumber\\
                  & \geq &\dfrac{1}{8(p+1)!}\left[\dfrac{1}{M_{\nu}D_{0}^{p+\alpha}}\right]^{\frac{1}{p+\alpha-1}}\left(f(x_{k+1})-f(x^{*})\right)^{\frac{p+\alpha}{p+\alpha-1}},
\label{eq:3.18}
\end{eqnarray}
where the last inequality is due to the convexity of $f$. Now, denoting
\begin{equation*}
\delta_{k}=\dfrac{f(x_{k})-f(x^{*})}{[8(p+1)!]^{p+\alpha-1}M_{\nu}D_{0}^{p+\alpha}}
\end{equation*}
we see from (\ref{eq:3.18}) that this sequence satisfies condition (1.1) of Lemma 1.1 in \cite{GN} with $u=\frac{p+\alpha}{p+\alpha-1}$. Note that $m$ is the first iteration for which $\delta_{m}\leq 2$. Using Lemma 1.1 in \cite{GN}, this inequality allow us to obtain the upper bound (\ref{eq:3.13}) for $m$ and also simplifies our final bound for the functional residual. Indeed, if $m>0$, then $\delta_{0}>2$ and, in view of inequality (1.2) of Lemma 1.1 in \cite{GN}, we have
\begin{equation*}
\ln\,2\leq \ln\,\delta_{m-1}\leq\left(\dfrac{p+\alpha-1}{p+\alpha}\right)^{m-1}\ln\,\delta_{0}\Longrightarrow \left(\dfrac{p+\alpha}{p+\alpha-1}\right)^{m-1}\ln\,2\leq\ln\,\delta_{0}
\end{equation*}
\begin{equation*}
\Longrightarrow \left(\dfrac{p+\alpha}{p+\alpha-1}\right)^{m-1}\leq\dfrac{\ln\,\delta_{0}}{\ln\,2}=\log_{2}\,\delta_{0}.
\end{equation*}
Thus, $ m\leq
\dfrac{\ln\,\delta_{0}}{\ln\left(\frac{p+\alpha}{p+\alpha-1}\right)}$,
and so (\ref{eq:3.13}) holds. Consequently, from
inequality (1.3) of Lemma 1.1 in \cite{GN} we get the
following rate of convergence:
\begin{equation*}
\delta_{k}\leq\left[\dfrac{1+\delta_{m}^{u-1}}{(u-1)(k-m)}\right]^{\frac{1}{u-1}}
\end{equation*}
that is,
\begin{equation*}
\dfrac{f(x_{k})-f(x^{*})}{[8(p+1)!]^{p+\alpha-1}M_{\nu}D_{0}^{p+\alpha}}\leq\left[\dfrac{(p+\alpha-1)(1+2^{\frac{1}{p+\alpha-1}})}{k-m}\right]^{p+\alpha-1}.
\end{equation*}
Therefore,
\begin{eqnarray*}
f(x_{k})-f(x^{*})&\leq &\dfrac{[8(1+2^{\frac{1}{p+\alpha-1}})(p+\alpha-1)(p+1)!]^{p+\alpha-1}M_{\nu}D_{0}^{p+\alpha}}{(k-m)^{p+\alpha-1}}\\
&\leq &
\dfrac{[24p(p+1)!]^{p+\alpha-1}M_{\nu}D_{0}^{p+\alpha}}{(k-m)^{p+\alpha-1}}.
\end{eqnarray*}
\end{proof}

If we assume that $\nu$ and $H_{f,p}(\nu)$ are known, by Lemma \ref{lem:extraA.2}, we can set
\begin{equation*}
M_{t}=M_{\nu}\equiv\max\left\{\dfrac{3H_{f,p}(\nu)}{2},3\theta (p-1)!\right\}.
\end{equation*}
Here, by (\ref{eq:3.1}) the corresponding version of Algorithm 1 takes at most \\ $\mathcal{O}(\epsilon^{-1/(p+\nu-1)})$ iterations to generate $x_{k}$ such that $f(x_{k})-f(x^{*})\leq \epsilon$ for a given $\epsilon\in (0,1)$. However, in most practical problems, $H_{f,p}(\nu)$ is not known. To deal with this situation, we can consider the following adaptive version of Algorithm 1:
\\[0.2cm]
\begin{mdframed}
\noindent\textbf{Algorithm 2. Adaptive Tensor Method}
\\[0.2cm]
\noindent\textbf{Step 0.} Choose $x_{0}\in\E$, $H_{0}>0$ and $\theta\geq 0$.  Set $\alpha$ by (\ref{eq:3.1}) and $t:=0$.\\
\noindent\textbf{Step 1.} Set $i:=0$.\\
\noindent\textbf{Step 1.1} Compute an approximate solution $x_{t,i}^{+}$ to
$
\min\limits_{y\in\E}\,\Omega^{(\alpha)}_{x_{t},p,2^{i}H_{t}}(y)$,
such that
\begin{equation*}
\Omega^{(\alpha)}_{x_{t},p,2^{i}H_{t}}(x_{t,i}^{+})\leq f(x_{t})\quad\text{and}\quad \|\nabla \Omega^{(\alpha)}_{x_{t},p,2^{i}H_{t}}(x_{t,i}^{+})\|_{*}\leq\theta\|x_{t,i}^{+}-x_{t}\|^{p+\alpha-1}.
\end{equation*}
\noindent\textbf{Step 1.2.} If
\begin{equation*}
f(x_{t})-f(x_{t,i}^{+})\geq \dfrac{1}{8(p+1)!(2^{i}H_{t})^{\frac{1}{p+\alpha-1}}}\|\nabla f(x_{t,i}^{+})\|_{*}^{\frac{p+\alpha}{p+\alpha-1}}
\end{equation*}
holds, set $i_{t}:=i$ and go to Step 2. Otherwise, set $i:=i+1$ and go to Step 1.1.\\
\noindent\textbf{Step 2.} Set $x_{t+1}=x_{t,i_{t}}^{+}$ and $H_{t+1}=2^{i_{t}-1}H_{t}$.\\
\noindent\textbf{Step 3.} Set $t:=t+1$ and go to Step 1.
\end{mdframed}

Let us define the following function of $\epsilon>0$:
\begin{equation}
N_{\nu}(\epsilon)=\left\{\begin{array}{ll}
\max\left\{\dfrac{3H_{f,p}(\nu)}{2},3\theta
(p-1)!\right\},&\text{if}\,\,\alpha = \nu,\\
\max\left\{\theta,\left(\dfrac{3H_{f,p}(\nu)}{2}\right)^{\frac{p}{p+\nu-1}}
\left(\dfrac{4R(\epsilon)}{\epsilon}\right)^{\frac{1-\nu}{p+\nu-1}}\right\},
&\text{if}\,\,\alpha = 1,
                         \end{array}
                  \right.
\label{eq:3.6}
\end{equation}
where
\begin{equation}
R(\epsilon)=\max_{x\in\E}\left\{\|x-x^{*}\|\,:\,f(x)\leq f(x^{*})+\epsilon\right\}.
\label{eq:3.7}
\end{equation}
The next lemma provides upper bounds on $H_{t}$ and on the
number of calls of the oracle\footnote{By \textit{calls of the oracle} we mean the joint computation of $f$ and its derivatives.}.

\begin{lemma}
\label{lem:3.1}
Suppose that H1 and H2 are true. Given $\epsilon>0$, assume that $\left\{x_{t}\right\}_{t=0}^{T}$ is a sequence generated by Algorithm 2 such that
\begin{equation}
f(x_{0})-f(x^{*})\geq\epsilon,
\label{eq:3.8}
\end{equation}
\begin{equation}
f(x_{t,i}^{+})-f(x^{*})\geq\epsilon,\quad i=0,\ldots,i_{t}\,\,\text{and}\,\, t=0,\ldots,T.
\label{eq:3.9}
\end{equation}
Then,
\begin{equation}
H_{t}\leq\max\left\{H_{0},N_{\nu}(\epsilon)\right\},\quad\text{for}\,\,t=0,\ldots,T.
\label{eq:3.10}
\end{equation}
Moreover, the number $O_{T}$ of calls of the oracle after $T$ iterations is bounded as follows:
\begin{equation}
O_{T}\leq 2T+\log_{2}\max\left\{H_{0},N_{\nu}(\epsilon)\right\}-\log_{2}H_{0}.
\label{eq:3.11}
\end{equation}
\end{lemma}

\begin{proof}
Let us prove (\ref{eq:3.10}) by induction. Clearly it
holds for $t=0$. Assume that (\ref{eq:3.10}) is true for
some $t$, $0\leq t\leq T-1$. If $\nu$ is known, then by
(\ref{eq:3.1}) we have $\alpha=\nu$. Thus, by H1 and Lemma
\ref{lem:extraA.2} the final value of $2^{i_{t}}H_{t}$
cannot exceed
\begin{equation*}
2\max\left\{\dfrac{3H_{f,p}(\nu)}{2},3\theta (p-1)!\right\},
\end{equation*}
since otherwise we should stop the line search earlier. Therefore,
\begin{equation*}
H_{t+1}=\frac{1}{2}2^{i_{t}}H_{t}\leq\max\left\{\dfrac{3H_{f,p}(\nu)}{2},3\theta (p-1)!\right\}=N_{\nu}(\epsilon)\leq\max\left\{H_{0},N_{\nu}(\epsilon)\right\},
\end{equation*}
that is, (\ref{eq:3.10}) holds for $t=t+1$.

On the other hand, if $\nu$ is unknown, we have $\alpha=1$. In view of (\ref{eq:3.7}), (\ref{eq:3.8}) and H2, it follows that
\begin{equation}
R(\epsilon)\leq R(f(x_{0})-f(x^{*}))=\max_{x\in\mathcal{L}(x_{0})}\|x-x^{*}\|\leq D_{0}<+\infty.
\label{eq:3.12}
\end{equation}
Thus, by (\ref{eq:3.9}) and Lemma A.5 in \cite{GN2} we have
$
\|\nabla f(x_{t+1})\|_{*}\geq\dfrac{\epsilon}{R(\epsilon)}.
$
In this case, it follows from Corollary \ref{cor:A.5} with $\delta=\epsilon/R(\epsilon)$ that
\begin{equation*}
2^{i_{t}}H_{t}\leq 2\max\left\{\theta,\left(\dfrac{3H_{f,p}(\nu)}{2}\right)^{\frac{p}{p+\nu-1}}\left(\dfrac{4R(\epsilon)}{\epsilon}\right)^{\frac{1-\nu}{p+\nu-1}}\right\}=2N_{\nu}(\epsilon).
\end{equation*}
Consequently, we also have
\begin{equation*}
H_{t+1}=\frac{1}{2}2^{i_{t}}H_{t}\leq N_{\nu}(\epsilon)\leq\max\left\{H_{0},N_{\nu}(\epsilon)\right\},
\end{equation*}
that is, (\ref{eq:3.10}) holds for $t+1$. This completes the induction argument.

Finally, note that at the $k$th iteration of Algorithm 2, the oracle is called $i_{k}+1$ times. Since $H_{k+1}=2^{i_{k}-1}H_{k}$, it follows that $i_{k}-1=\log_{2}H_{k+1}-\log_{2}H_{k}$. Thus, by (\ref{eq:3.10}) we get
\begin{eqnarray*}
O_{T}&=&\sum_{k=0}^{T-1}(i_{k}+1)=\sum_{k=0}^{T-1}\,2+\log_{2}H_{k+1}-\log_{2}H_{k}=2T+\log_{2}H_{T}-\log_{2}H_{0}\\
&\leq & 2T+\log_{2}\max\left\{H_{0},N_{\nu}(\epsilon)\right\}-\log_{2}H_{0}.
\end{eqnarray*}
\end{proof}

From Lemma \ref{lem:3.1}, we see that Algorithm 2 is a particular case of Algorithm 1 in which
\begin{equation*}
M_{t}=2^{i_{t}}H_{t}=2H_{t+1}\quad\text{and}\quad M_{\nu}=2\max\left\{H_{0},N_{\nu}(\epsilon)\right\}.
\end{equation*}
Thus, combining Theorem 3.2 and Lemma 3.3, we obtain the following result.
\begin{theorem}
\label{thm:3.2} Suppose that H1 and H2 are true. Given
$\epsilon\in (0,1)$, assume that
$\left\{x_{t}\right\}_{t=0}^{T}$ is a sequence generated
by Algorithm 2 such that 
\begin{equation}
f(x_{0})-f(x^{*})\geq\epsilon,
\label{eq:3.8plus}
\end{equation}
\begin{equation}
f(x_{t,i}^{+})-f(x^{*})\geq\epsilon,\quad i=0,\ldots,i_{t}\,\,\text{and}\,\, t=0,\ldots,T.
\label{eq:3.9plus}
\end{equation}
Denote by $m$ the first iteration number such that
\begin{equation*}
f(x_{m})-f(x^{*})\leq [8(p+1)!]^{p+\alpha-1}4\max\left\{H_{0},N_{\nu}(\epsilon)\right\}D_{0}^{p+\alpha},
\end{equation*}
and assume that $m<T$. Then
\begin{equation}
m\leq\dfrac{1}{\ln\left(\frac{p+\alpha}{p+\alpha-1}\right)}\ln\max\left\{1,\log_{2}\dfrac{f(x_{0})-f(x^{*})}{[8(p+1)!]^{p+\alpha-1}2\max\left\{H_{0},N_{\nu}(\epsilon)\right\}D_{0}^{p+\alpha}}\right\}
\label{eq:3.19}
\end{equation}
and
\begin{equation}
f(x_{T})-f(x^{*})\leq\dfrac{[24p(p+1)!]^{p+\alpha-1}2\max\left\{H_{0},N_{\nu}(\epsilon)\right\}D_{0}^{p+\alpha}}{(T-m)^{p+\alpha-1}}
\label{eq:3.20}
\end{equation}
Consequently,
\begin{equation}
T\leq m+\kappa_{1}^{(\nu)}[24p(p+1)!]\epsilon^{-\frac{1}{p+\nu-1}},
\label{eq:3.21}
\end{equation}
where
\begin{equation*}
\kappa_{1}^{(\nu)}=\left\{\begin{array}{ll} \left(2\max\left\{H_{0},\dfrac{3H_{f,p}(\nu)}{2},3\theta (p-1)!\right\}D_{0}^{p+\nu}\right)^{\frac{1}{p+\nu-1}},&\text{if}\,\,\nu\,\,\text{is known},\\
                                          \left(2\max\left\{H_{0},\theta,\left(\dfrac{3H_{f,p}(\nu)}{2}\right)^{\frac{p}{p+\nu-1}}\left(4D_{0}\right)^{\frac{1-\nu}{p+\nu-1}}\right\}D_{0}^{p+1}\right)^{\frac{1}{p}},&\text{if}\,\,\nu\,\,\text{is unknown}.
                         \end{array}
                  \right.
\end{equation*}
\end{theorem}

\begin{proof}
As mentioned above, by Lemma 3.3 we have
\begin{equation*}
2^{i_{t}}H_{t}=2(2^{i_{t}-1}H_{t})=2H_{t+1}\leq 2\max\left\{H_{0},N_{\nu}(\epsilon)\right\},\,\,t=0,\ldots,T-1.
\end{equation*}
Then, (\ref{eq:3.19}) and (\ref{eq:3.20}) follow directly from Theorem \ref{thm:3.1} with
\begin{equation*}
M_{\nu}=2\max\left\{H_{0},N_{\nu}(\epsilon)\right\}.
\end{equation*}
Now, combining (\ref{eq:3.9plus}) and (\ref{eq:3.20}), we obtain
\begin{equation*}
\epsilon\leq\dfrac{[24p(p+1)!]^{p+\alpha-1}2\max\left\{H_{0},N_{\nu}(\epsilon)\right\}D_{0}^{p+\alpha}}{(T-m)^{p+\alpha-1}}
\end{equation*}
and so,
\begin{equation}
T\leq m+\dfrac{[24p(p+1)!]\left(2\max\left\{H_{0},N_{\nu}(\epsilon)\right\}D_{0}^{p+\alpha}\right)^{\frac{1}{p+\alpha-1}}}{\epsilon^{\frac{1}{p+\alpha-1}}}.
\label{eq:3.22}
\end{equation}
If $\nu$ is known, then $\alpha=\nu$ and, by (\ref{eq:3.6}), we have
\begin{equation}
N_{\nu}(\epsilon)=\max\left\{\dfrac{3H_{f,p}(\nu)}{2},3\theta (p-1)!\right\}.
\label{eq:3.23}
\end{equation}
Thus, combining (\ref{eq:3.22}) and (\ref{eq:3.23}), we get (\ref{eq:3.21}). On the other hand, if $\nu$ is unknown, then $\alpha=1$ and, by (\ref{eq:3.6}), (\ref{eq:3.12}) and $\epsilon\in (0,1)$, we have
\begin{eqnarray}
N_{\nu}(\epsilon)&=&\max\left\{\theta,\left(\dfrac{3H_{f,p}(\nu)}{2}\right)^{\frac{p}{p+\nu-1}}\left(\dfrac{4R(\epsilon)}{\epsilon}\right)^{\frac{1-\nu}{p+\nu-1}}\right\}\nonumber\\
&\leq &\max\left\{\theta,\left(\dfrac{3H_{f,p}(\nu)}{2}\right)^{\frac{p}{p+\nu-1}}\left(4D_{0}\right)^{\frac{1-\nu}{p+\nu-1}}\right\}\epsilon^{-\frac{1-\nu}{p+\nu-1}}.
\label{eq:3.24}
\end{eqnarray}
In this case, combining (\ref{eq:3.22}) and (\ref{eq:3.24}) we also get (\ref{eq:3.21}).
\end{proof}

\begin{remark}
Note that for any $\theta$ in the interval
\begin{equation*}
0<\theta<\left\{\begin{array}{ll} \min\left\{\dfrac{H_{0}}{3(p-1)!},\dfrac{H_{f,p}(\nu)}{(p-1)!2}\right\},&\text{if $\nu$ is known},\\
                                  \min\left\{H_{0},\left(\dfrac{3 H_{f,p}(\nu)}{2}\right)^{\frac{p}{p+\nu-1}}(4D_{0})^{\frac{1-\nu}{p+\nu-1}}\right\},&\text{if $\nu$ is unknown},
                \end{array}
         \right.
\end{equation*}
the corresponding right-hand side in (3.21) has the same value as for $\theta=0$.
\end{remark}

Note that Algorithm 2 with $\alpha=1$ is a universal
scheme: it works for any H\"{o}lder parameter $\nu\in
[0,1]$ without using it explicitly. This algorithm can be
viewed as a generalization of the universal method (6.10)
in \cite{GN}. Looking at the efficiency bound
(\ref{eq:3.21}), for $\nu$ known and $\nu$ unknown, we see
that the universal scheme ensures the same dependence on
the accuracy $\epsilon$ as the nonuniversal scheme
($\alpha=\nu\neq 1$). Remarkably, this is not true for the
accelerated schemes obtained from the standard estimating
sequences technique, as we will see in the next section.

\section{Accelerated tensor schemes}

Similarly to Section 3, we shall consider a general accelerated tensor method parametrized by the constant $\alpha$ given in (\ref{eq:3.1}). Specifically, at the beginning of the $t$th iteration ($t>0$) one has an estimate $x_{t}$ for the solution of (\ref{eq:2.1}), an auxiliary vector $v_{t}$ and constants $A_{t},M_{t}>0$. A new vector $y_{t}$ is computed as a convex combination of $x_{t}$ and $v_{t}$:
\begin{equation}
y_{t}=(1-\gamma_{t})x_{t}+\gamma_{t}v_{t},
\label{eq:4.1}
\end{equation}
where
\begin{equation}
\gamma_{t}=\dfrac{a_{t}}{A_{t}+a_{t}}
\label{eq:4.2}
\end{equation}
with $a_{t}>0$ being computed from the equation
\begin{equation}
a_{t}^{p+\alpha}=\dfrac{1}{2^{(3p-1)}}\left[\dfrac{(p-1)!}{M_{t}}\right]\left(A_{t}+a_{t}\right)^{p+\alpha-1}
\label{eq:4.3}
\end{equation}
Then, a trial point $x_{t}^{+}$ is computed as an approximate solution to the auxiliary problem
\begin{equation}
\min_{x\in\E}\,\Omega_{y_{t},p,M_{t}}^{(\alpha)}(x),
\label{eq:4.4}
\end{equation}
such that
\begin{equation}
\Omega_{y_{t},p,M_{t}}^{(\alpha)}(x_{t}^{+})\leq f(y_{t})\quad\text{and}\quad \|\nabla\Omega_{y_{t},p,M_{t}}^{(\alpha)}(x_{t}^{+})\|_{*}\leq\theta\|x_{t}^{+}-y_{t}\|^{p+\alpha-1},
\label{eq:4.5}
\end{equation}
where $\theta\geq 0$ is a user-defined parameter. If the descent condition
\begin{equation}
\langle\nabla f(x_{t}^{+}),y_{t}-x_{t}^{+}\rangle \geq \dfrac{1}{4}\left[\dfrac{(p-1)!}{M_{t}}\right]^{\frac{1}{p+\alpha-1}}\|\nabla f(x_{t}^{+})\|_{*}^{\frac{p+\alpha}{p+\alpha-1}}
\label{eq:4.7}
\end{equation}
is satisfied, then $x_{t}^{+}$ is accepted, and we define
$x_{t+1}=x_{t}^{+}$.
Otherwise, constant $M_{t}$ is increased until the
corresponding trial point $x_{t}^{+}$ is accepted. As in
Algorithm 1, we assume that there exists $M_{\nu}>0$ such
that $M_{t}\leq M_{\nu}$ for all $t$. After obtaining
$x_{t+1}$, we set $A_{t+1}=A_{t}+a_{t}$ and compute
\begin{equation}
v_{t+1}=\arg\min_{x\in\E}\,\psi_{t+1}(x),
\label{eq:4.8}
\end{equation}
where
\begin{equation}
\psi_{t+1}(x)=\psi_{t}(x)+a_{t}\left[f(x_{t+1})+\langle\nabla f(x_{t+1}),x-x_{t+1}\rangle\right].
\label{eq:4.9}
\end{equation}
To initialize, we choose $x_{0}$ and we set $v_{0}=x_{0}$, $A_{0}=0$, and $\psi_{0}(x)=\frac{1}{p+\alpha}\|x-x_{0}\|^{p+\alpha}$. This general scheme can be summarized in the following way.
\begin{mdframed}
\noindent\textbf{Algorithm 3. Accelerated Tensor Method}
\\[0.05cm]
\noindent\textbf{Step 0.} Choose $x_{0}\in\E$, $H_{0}>0$. Set $\alpha$ by (\ref{eq:3.1}), $v_{0}=x_{0}$, $A_{0}=0$, and $t:=0$.\\
\noindent\textbf{Step 1.} Find $0<M_{t}\leq M_{\nu}$ such that (\ref{eq:4.7}) holds for an approximate solution $x_{t}^{+}$ to (\ref{eq:4.4}) satisfying (\ref{eq:4.5}), with $y_{t}$ being defined by (\ref{eq:4.1})-(\ref{eq:4.3}).\\
\noindent\textbf{Step 2.} Set $x_{t+1}=x_{t}^{+}$ and $A_{t+1}=A_{t}+a_{t}$, with $a_{t}>0$ obtained from (\ref{eq:4.3}).\\
\noindent\textbf{Step 3.} Define $\psi_{t+1}(.)$ by (\ref{eq:4.9}) and compute $v_{t+1}$ by (\ref{eq:4.8}).\\
\noindent\textbf{Step 4.} Set $t:=t+1$ and go back to Step 1.
\end{mdframed}

\begin{remark}
From the expression of $\psi_{t+1}(\,.\,)$ we can see that $\min_{x\in\E}\psi_{t+1}(x)$ admits a closed form solution, namely,
\begin{equation*}
v_{t+1}=x_{0}-\dfrac{B^{-1}s_{t+1}}{\|s_{t+1}\|^{\frac{p+\alpha-2}{p+\alpha-1}}},\quad s_{t+1}=\sum_{i=0}^{t}a_{i}\nabla f(x_{i+1}).
\end{equation*}
Regarding the computation of $a_{t}$, for $t=0$, (\ref{eq:4.3}) gives
\begin{equation*}
a_{0}=\dfrac{1}{2^{(3p-1)}}\left[\dfrac{(p-1)!}{M_{t}}\right].
\end{equation*}
For $t>0$, we have $A_{t}>0$. Thus, the computation of $a_{t}$ requires the solution of a univariate nonlinear equation of the form
\begin{equation*}
x^{p+\alpha}-B(A+x)^{p+\alpha-1}=0,\quad B,A>0.
\end{equation*}
Denoting $g(x)=x^{p+\alpha}-B(A+x)^{p+\alpha-1}$, it is easy to see that
\begin{equation*}
g(0)=-BA^{p+\alpha-1}<0\leq g(c),
\end{equation*}
where $c=\max\left\{A,2^{p+\alpha-1}B\right\}$. Since $g(\,.\,)$ is continuous, we can use the bisection Method to compute an approximation to $a^{*}\in (0,c]$ such that $g(a^{*})=0$. As can be seen in the proof of Theorem \ref{thm:B4.2}, our convergence results only require 
\begin{equation*}
0<a_{t}^{p+\alpha}\leq\dfrac{1}{2^{(3p-1)}}\left[\dfrac{(p-1)!}{M_{t}}\right](A_{t}+a_{t})^{p+\alpha-1}.
\end{equation*}
\end{remark}

The next result establishes the relationship between the
estimating functions $\psi_{t}(\cdot)$ and the objective
function $f(\cdot)$.
\begin{lemma}
\label{lem:4.1}
For all $t\geq 0$,
\beq\label{eq:4.10}
\ba{rcl}
\psi_{t}(x) & \leq & A_{t}f(x)+{1\over
(p+\alpha)}\|x-x_{0}\|^{p+\alpha},\,\,\forall x\in\mathbb{E}.
\ea
\eeq
\end{lemma}

\begin{proof}
We prove this result by induction in $t$. Since $A_0 =
0$, for all $x\in\mathbb{E}$,
$$
\ba{rcl}
\psi_{0}(x)&=&{1 \over (p+\alpha)}\|x-x_{0}\|^{p+\alpha} =
A_{0}f(x)+{1 \over (p+\alpha)}\|x-x_{0}\|^{p+\alpha},
\ea
$$
that is, (\ref{eq:4.10}) is true for $t=0$. Suppose that
(\ref{eq:4.10}) is true for some $t\geq 0$. Then 
(\ref{eq:4.9}) and the convexity of $f$ imply that, for all
$x\in\mathbb{E}$,
$$
\ba{rcl}
\psi_{t+1}(x)& =  &\psi_{t}(x)+a_{t}\left[f(x_{t+1})+ \la
\nabla f(x_{t+1}),x-x_{t+1}\ra\right]\\
&\leq&\psi_{t}(x)+a_{t}f(x)
\\
&\leq& (A_{t}+a_{t})f(x)+{\|x-x_{0}\|^{p+\alpha} \over
(p+\alpha)}
\\
&  = &A_{t+1}f(x)+{\|x-x_{0}\|^{p+\alpha}\over (p+\alpha)}.
\ea
$$
Thus, (\ref{eq:4.10}) is also true for $t+1$, and the
proof is completed.
\end{proof}

The theorem below establishes the global convergence rate
for Algorithm 3.
%
\begin{theorem}
\label{thm:B4.2} Assume that H1 is true and let the sequence
$\left\{x_{t}\right\}_{t=0}^{T}$ be generated by
Algorithm 3. Then, for $t=2,\ldots,T$,
\begin{equation}
f(x_{t})-f(x^{*})\leq\dfrac{2^{3p-1}M_{\nu}(p+\alpha)^{p+\alpha-1}\|x_{0}-x^{*}\|^{p+\alpha}}{(p-1)!(t-1)^{p+\alpha}}.
\label{eq:4.13}
\end{equation}

\end{theorem}

\begin{proof}
Let us prove by induction that
\begin{equation}
A_{t}f(x_{t})\leq\psi_{t}^{*}\equiv\min_{x\in\E}\psi_{t}(x).
\label{eq:4.16}
\end{equation}
Since $A_{0}=0$, we have $A_{0}f(x_{0})=0=\min_{x\in E}\psi_{0}(x)$. Thus, (\ref{eq:4.16}) is true for $t=0$. Assume that it is true for some $t\geq 0$. Note that for any $x\in E$ we have
$$
\ba{rcl}
\psi_{t}(x)&=&\sum_{i=0}^{t-1}a_{i}\left[f(x_{i+1})+ \la
\nabla f(x_{i+1}),x-x_{i+1}\ra\right]+{\|x-x_{0}\| \over p+\alpha}^{p+\alpha}\\
&\equiv&\ell_{t}(x)+{1\over
p+\alpha}\|x-x_{0}\|^{p+\alpha},\,\,\text{for all}\,\, t\geq 1.
\ea
$$
Note that $\ell_{t}(x)$ is a linear function. Moreover, by
Lemma 4 in \cite{NES3}, function \\ ${1\over
(p+\alpha)}\|x-x_{0}\|^{p+\alpha}$ is uniformly convex of degree
$p+\alpha$ with parameter $2^{-(p+\alpha-2)}$. Thus,
$\psi_{t}(x)$ is also a uniformly convex function of
degree $p+\alpha$ with parameter
$2^{-(p+\alpha-2)}$. Therefore, Lemma A.2 in \cite{GN2}
and the induction assumption imply that
\begin{equation*}
\psi_{t}(x)\geq 
\psi_{t}^{*}+{2^{-(p+\alpha-1)}
\over p+\alpha}\|x-v_{t}\|^{p+\alpha}\geq  A_{t}f(x_{t})+{2^{-(p+\alpha-1)}\over p+\alpha}\|x-v_{t}\|^{p+\alpha}.
\end{equation*}
Thus,
$$
\ba{rcl}
\psi_{t+1}^{*}&  = &\min\limits_{x\in\E}\left\{\psi_{t}(x)+a_{t}\left[f(x_{t+1})+\la \nabla f(x_{t+1}),x-x_{t+1}\ra\right]\right\}\nonumber\\
\\
&\geq&\min\limits_{x\in\E}
\{A_{t}f(x_{t})+{2^{-(p+\alpha-2)}
\over
(p+\alpha)}\|x-v_{t}\|^{p+\alpha}\\
\\
& & +a_{t}[f(x_{t+1})+\langle \nabla
f(x_{t+1}),x-x_{t+1}\rangle]\}.
\ea
$$
Since $f$ is convex and differentiable, we have
$$
\ba{rcl}
f(x_{t}) &\geq &f(x_{t+1})+\la \nabla
f(x_{t+1}),x_{t}-x_{t+1}\ra.
\ea
$$
Then, substituting this inequality above, we obtain
$$
\ba{rcl}
\psi_{t+1}^{*} & \geq & \min\limits_{x\in\E}\{A_{t+1}f(x_{t+1}) +\la \nabla
f(x_{t+1}),A_{t}x_{t}-A_{t}x_{t+1}\ra\\
\\
& & +a_{t}\la \nabla f(x_{t+1}),x-x_{t+1}\ra+
{2^{-(p+\alpha-1)}\over (p+\alpha)}\|x-v_{t}\|^{p+\alpha}.
\ea
$$
Note that $y_{t}=(1-\gamma_{t})x_{t}+\gamma_{t}v_{t}=
{A_{t} \over A_{t+1}}x_{t}+{a_{t}\over A_{t+1}}v_{t}$.
Therefore, $A_{t}x_{t}=A_{t+1}y_{t}-a_{t}v_{t}$, and
$$
\ba{rcl}
\psi_{t+1}^{*} & \geq & \min\limits_{x\in\E}\{A_{t+1}f(x_{t+1}) +\la \nabla
f(x_{t+1}),A_{t+1}y_{t}-a_{t}v_{t}-A_{t}x_{t+1}
\ra\\
\\
& &  +a_{t}\la \nabla f(x_{t+1}),x-x_{t+1}\ra
+{2^{-(p+\alpha-1)}\over (p+\alpha)}\|x-v_{t}\|^{p+\alpha}.
\ea
$$
Moreover, $A_{t+1}x_{t+1}=A_{t}x_{t+1} +a_{t}x_{t+1}$, and so
$$
\ba{rcl}
\psi_{t+1}^{*}&\geq&\min\limits_{x\in\E}\{A_{t+1} f(x_{t+1})+A_{t+1}\la
\nabla f(x_{t+1}),y_{t}-x_{t+1}\ra\\
\\
& & +a_{t}\la \nabla f(x_{t+1}),x-v_{t}
\ra+{2^{-(p+\alpha-1)}\over (p+\alpha)}\|x-v_{t}\|^{p+\alpha}\\
&\geq& A_{t+1}f(x_{t+1})+\min\limits_{x\in\E} \{A_{t+1}{1\over 4}\left[{(p-1)! \over M_{t}}\right]^{ {1\over
p+\alpha-1}} \|\nabla
f(x_{t+1})\|_{*}^{ {p+\alpha\over p+\alpha-1}}\\
\\
& & +a_{t}\la \nabla f(x_{t+1}),x-v_{t}\ra+
{2^{-(p+\alpha-1)}\over (p+\alpha)}\|x-v_{t}\|^{p+\alpha}\},
\ea
$$
where the last inequality is due to (\ref{eq:4.7}). Thus,
to prove that (\ref{eq:4.16}) is true for $t+1$, it is
enough to show that
\beq\label{eq:4.17}
\ba{c}
A_{t+1}{1\over 4}\left[{(p-1)! \over M_{t}}\right]^{ {1\over
p+\alpha-1}} \|\nabla
f(x_{t+1})\|_{*}^{ {p+\alpha\over p+\alpha-1}}+a_{t} \la
\nabla f(x_{t+1}),x-v_{t}\ra\\
+ {2^{-(p+\alpha-1)}\over (p+\alpha)}\|x-v_{t}\|^{p+\alpha}\}\geq 0
\ea
\eeq
for all $x\in\mathbb{E}$. Using Lemma 2 in \cite{NES3} with
$r=p+\nu$, $s=a_{t}\nabla f(x_{t+1})$ and
$\omega=2^{-(p+\alpha-1)}$, we see that a
sufficient condition for (\ref{eq:4.17}) is
$$
\ba{rcl}
A_{t+1}{1\over 4}\left[{(p-1)! \over M_{t}}\right]^{ {1\over
p+\alpha-1}} \|\nabla
f(x_{t+1})\|_{*}^{ {p+\alpha\over p+\alpha-1}} &
\geq & {\left(p+\alpha-1\over p+\alpha\right)}2^{ {p+\alpha-2\over p+\alpha-1}}a_{t}^{
{p+\alpha\over p+\alpha-1}}\|\nabla f(x_{t+1})\|_{*}^{{p+\alpha\over p+\alpha-1}},
\ea
$$
which is equivalent to
$$
\ba{rcl}
a_{t}^{p+\alpha} & \leq &
2\left({p+\alpha \over p+\alpha-1}\right)^{p+\alpha-1}
{\left(1\over 8\right)}^{p+\alpha-1}\left[{(p-1)!\over M_{t}}\right]A_{t+1}^{p+\alpha-1}.
\ea
$$
Note that,
$
2\left({p+\alpha \over p+\alpha-1}\right)^{p+\alpha-1}
\left(1 \over 8\right)^{p+\alpha-1}\geq {1\over 2^{(3p-1)}}
$. Therefore, by (\ref{eq:4.3}) we have

\begin{eqnarray*}
a_{t}^{p+\alpha}&=&\dfrac{1}{2^{(3p-1)}}\left[\dfrac{(p-1)!}{M_{t}}\right]\left(A_{t}+a_{t}\right)^{p+\alpha-1}\\
&\leq&
2\left({p+\alpha \over p+\alpha-1}\right)^{p+\alpha-1}
{\left(1\over 8\right)}^{p+\alpha-1}\left[{(p-1)!\over M_{t}}\right]A_{t+1}^{p+\alpha-1}.
\end{eqnarray*}
Thus (\ref{eq:4.16}) is true for $t+1$, completing the induction argument.

Let us now estimate the growth of the coefficients
$A_{t}$. Since $M_{t}\leq M_{\nu}$ for all $t=0,\ldots,T$, by (\ref{eq:4.3}) we get $
a_{t}^{p+\alpha}\geq\dfrac{1}{\tilde{M}}(A_{t}+a_{t})^{p+\alpha-1}
$ with
\begin{equation}
\tilde{M}={2^{(3p-1)}M_{\nu}\over (p-1)!}.
\label{eq:4.18}
\end{equation}
Consequently,
\beq\label{eq:4.19}
\ba{rcl}
A_{t+1}-A_{t} & = & a_t \; \geq \;
\left({1 \over \tilde{M}}\right)^{ {1\over p+\alpha}} A_{t+1}^{ {p+\alpha-1 \over
p+\alpha}}.
\ea
\eeq
Now, denoting $B_{t}=\tilde{M}A_{t}$ for all
$t\geq 0$, it follows from (\ref{eq:4.19}) that
$$
\ba{rcl}
B_{t+1}-B_{t} & \geq & B_{t+1}^{ {p+\alpha-1\over p+\alpha}}.
\ea
$$
Then, by Lemma A.4 in \cite{GN2}, 
we have
$$
\ba{rcl}
B_{t} & \geq &
\left[\left({1 \over p+\alpha}\right)\left(
{B_{1}^{ {1\over p+\alpha}}\over B_{1}^{ {1\over
p+\alpha}}+1}\right)^{ {p+\alpha-1\over p+\alpha}}
\right]^{p+\alpha}(t-1)^{p+\alpha} \quad \forall t\geq 2.
\ea
$$
Note that $A_{1}\geq {1 \over 2M}$. Thus,
$B_{1}\geq 1$ and consequently
$$
\ba{rcl}
B_{t} & \geq &
\left[{1\over (p+\alpha)}\left(\half\right)^{ {p+\alpha-1\over p+\alpha}}
\right]^{p+\alpha}(t-1)^{p+\alpha}.
\ea
$$
Therefore, for all $t\geq 2$, we have
\begin{equation}
A_{t}\geq {1\over \tilde{M}}\left[{1\over (p+\alpha)}\left(\dfrac{1}{2}\right)^{
{p+\alpha-1\over p+\alpha}} \right]^{p+\alpha}(t-1)^{p+\alpha}.
\label{eq:4.20}
\end{equation}
Finally, by (\ref{eq:4.16}) and Lemma \ref{lem:4.1}, for
$t\geq 0$, we have
$$
\ba{rcl}
A_{t}f(x_{t}) & \leq & \psi_{t}^{*} \; \leq \;
A_{t}f(x^{*})+{1 \over
p+\alpha}\|x^{*}-x_{0}\|^{p+\alpha}.
\ea
$$
Hence, $A_{t}(f(x_{t})-f(x^{*}))\leq{1
\over 2+\nu}\|x^{*}-x_{0}\|^{2+\nu}$, and (\ref{eq:4.13})
follows immediately from (\ref{eq:4.18}) and (\ref{eq:4.20}).
\end{proof}

If we assume that $\nu$ and $H_{f,p}(\nu)$ are known, then, by Lemma \ref{lem:A.6}, we can set
\begin{equation*}
M_{t}=M_{\nu}\equiv (p+\nu-1)(H_{f,p}(\nu)+\theta (p-1)!).
\end{equation*}
Here, by (\ref{eq:3.1}) the corresponding version
of Algorithm 3 takes at most\\
$\mathcal{O}(\epsilon^{-1/(p+\nu)})$ iterations to
generate $x_{t}$ such that $
f(x_{t})-f(x^{*})\leq\epsilon. $ For problems in which
$H_{f,p}(\nu)$ is not known, let us consider the following
adaptive version of Algorithm 3.

\begin{mdframed}
\noindent\textbf{Algorithm 4. Adaptive Accelerated Tensor Method}
\\[0.2cm]
\noindent\textbf{Step 0.} Choose $x_{0}\in\E$, $H_{0}>0$, and $\theta\geq 0$. Set $\alpha$ by (\ref{eq:3.1}) and define function $\psi_{0}(x)=\frac{1}{p+\alpha}\|x-x_{0}\|^{p+\alpha}$. Set $v_{0}=x_{0}$, $A_{0}=0$, and $t:=0$.\\
\noindent\textbf{Step 1.} Set $i:=0$.\\
\noindent\textbf{Step 1.1.} Compute the coefficient $a_{t,i}>0$ by solving equation
\begin{equation*}
a_{t,i}^{p+\alpha}=\dfrac{1}{2^{(3p-1)}}\left[\dfrac{(p-1)!}{2^iH_{t}}\right](A_{t}+a_{t,i})^{p+\alpha-1}.
\end{equation*}
\noindent\textbf{Step 1.2.} Set $\gamma_{t,i}=\dfrac{a_{t,i}}{A_{t}+a_{t,i}}$ and compute vector $y_{t,i}=(1-\gamma_{t,i})x_{t}+\gamma_{t,i}v_{t}$.\\
\noindent\textbf{Step 1.3} Compute an approximate solution $x_{t,i}^{+}$ to
$
\min_{x\in\E}\Omega_{y_{t,i},p,2^{i}H_{t}}^{(\alpha)}(x),
$
such that
\begin{equation*}
\Omega_{y_{t,i},p,2^{i}H_{t}}^{(\alpha)}(x_{t,i}^{+})\leq f(y_{t,i})\quad\text{and}\quad \|\nabla\Omega_{y_{i,t},p,2^{i}H_{t}}^{(\alpha)}(x_{t,i}^{+})\|_{*}\leq\theta\|x_{t,i}^{+}-y_{t,i}\|^{p+\alpha-1}.
\end{equation*}
\noindent\textbf{Step 1.4.} If condition
\begin{equation*}
\la \nabla f(x_{t,i}^{+}),y_{t,i}-x_{t,i}^{+}\ra\geq\dfrac{1}{4}\left[\dfrac{(p-1)!}{2^{i}H_{t}}\right]^{\frac{1}{p+\alpha-1}}\|\nabla f(x_{t,i}^{+})\|_{*}^{\frac{p+\alpha}{p+\alpha-1}},
\end{equation*}
set $i_{t}:=i$ and go to Step 2. Otherwise, set $i:=i+1$ and go back to Step 1.1.\\
\noindent\textbf{Step 2.} Set $x_{t+1}=x_{t,i_{t}}^{+}$, $y_{t}=y_{t,i_{t}}$, $a_{t}=a_{t,i_{t}}$ and $\gamma_{t}=\gamma_{t,i_{t}}$. Define $A_{t+1}=A_{t}+a_{t}$ and $H_{t+1}=2^{i_{t}-1}H_{t}$.\\
\noindent\textbf{Step 3.} Define $\psi_{t+1}(.)$ by (\ref{eq:4.9}) and compute $v_{t+1}$ by (\ref{eq:4.8}).\\
\noindent\textbf{Step 4.} Set $t:=t+1$ and go back to Step 1.
\end{mdframed}

Note that Algorithm 4 is a particular case of Algorithm 3 in which
\begin{equation*}
M_{t}=2^{i}H_{t}\quad \forall t\geq 0.
\end{equation*}
Let us define the following function of $\epsilon>0$:
\begin{equation}
\tilde{N}_{\nu}(\epsilon)=\left\{\begin{array}{ll}
(p+\nu-1)(H_{f,p}(\nu)+\theta(p-1)!),&\text{if}\,\,\alpha = \nu,\\
\max\left\{4\theta(p-1)!,(4H_{f,p}(\nu))^{\frac{p}{p+\nu-1}}
\left(\dfrac{4R(\epsilon)}{\epsilon}\right)^{\frac{1-\nu}{p+\nu-1}}\right\},
&\text{if}\,\,\alpha = 1.
                                    \end{array}
                             \right.
\label{eq:4.21}
\end{equation}

The next lemma provides upper bounds on $H_{t}$ and on the number of calls of the oracle in Algorithm 4.

\begin{lemma}
\label{lem:B4.1}
Suppose that H1 and H2 are true. Given $\epsilon>0$, assume that $\left\{x_{t}\right\}_{t=0}^{T}$ is a sequence generated by Algorithm 4 such that
\begin{equation}
f(x_{0})-f(x^{*})\geq\epsilon
\label{eq:4.22}
\end{equation}
and
\begin{equation}
f(x_{t,i}^{+})-f(x^{*})\geq\epsilon,\quad  i=0,\ldots,i_{t}\,\,\text{and}\,\, t=0,\ldots,T.
\label{eq:4.23}
\end{equation}
Then
\begin{equation}
H_{t}\leq\max\left\{H_{0},\tilde{N}_{\nu}(\epsilon)\right\}\quad\text{for}\,\,t=0,\ldots,T.
\label{eq:4.24}
\end{equation}
Moreover, the number $O_{T}$ of calls of the oracle after $T$ iterations is bounded as follows:
\begin{equation}
O_{T}\leq 2T+\log_{2}\max\left\{H_{0},\tilde{N}_{\nu}(\epsilon)\right\}-\log_{2}H_{0}.
\label{eq:4.25}
\end{equation}
\end{lemma}

\begin{proof}
Let us prove by induction that the scaling coefficients $H_{t}$ in Algorithm 4 satisfy (\ref{eq:4.24}). This is obvious for $t=0$. Assume that (\ref{eq:4.24}) is true for some $t\geq 0$. If $\alpha=\nu$, it follows from Lemma \ref{lem:A.6} that the final value $2^{i_{t}}H_{t}$ cannot be bigger than
\begin{equation*}
2\left[(p+\nu-1)(H_{f,p}(\nu)+\theta(p-1)!)\right],
\end{equation*}
since otherwise we should stop the line-search earlier. Thus,
\begin{equation*}
H_{t+1}=\dfrac{1}{2}2^{i_{t}}H_{t}\leq (p+\nu-1)(H_{f,p}(\nu)+\theta(p-1)!)\leq\max\left\{\tilde{N}_{\nu}(\epsilon),H_{0}\right\},
\end{equation*}
that is, (\ref{eq:4.25}) holds for $t+1$. On the other hand, suppose that $\alpha=1$. In view of Lemma A.5 in \cite{GN2}, at any trial point $x_{t,i}^{+}$ we have
\begin{equation*}
\|\nabla f(x_{t,i}^{+})\|_{*}\geq\dfrac{\epsilon}{R(\epsilon)}.
\end{equation*}
Thus, it follows from Lemma \ref{lem:A.7} that
\begin{equation*}
2^{i_{t}}H_{t}\leq 2\max\left\{4\theta(p-1)!,(4H_{f,p}(\nu))^{\frac{p}{p+\nu-1}}\left(\dfrac{4R(\epsilon)}{\epsilon}\right)^{\frac{1-\nu}{1+\nu}}\right\}\leq 2\max\left\{\tilde{N}_{\nu}(\epsilon),H_{0}\right\}.
\end{equation*}
Consequently, we also have $H_{t+1}\leq\max\left\{M_{\nu}(\epsilon),H_{0}\right\}$; i.e., (\ref{eq:4.24}) holds for $t+1$. This completes the induction argument. Finally, as in the proof of Lemma 3.3, from (\ref{eq:4.24}) we get (\ref{eq:4.25}).
\end{proof}

Now we can prove the following convergence result for
Algorithm 4.

\begin{theorem}
Suppose that H1 and H2 are true. Given $\epsilon\in(0,1)$,
assume that $\left\{x_{t}\right\}_{t=0}^{T}$ is a sequence
generated by Algorithm such that (\ref{eq:4.22}) and
(\ref{eq:4.23}) hold. Then
\begin{equation}
f(x_{t})-f(x^{*})\leq\dfrac{2^{3p}\max\left\{\tilde{N}_{\nu}
(\epsilon),H_{0}\right\}(p+\alpha)^{p+\alpha-1}\|x_{0}-x^{*}\|^{p+\alpha}}
{(p-1)!(t-1)^{p+\alpha}}, \quad 2 \leq t \leq T.
\label{eq:4.26}
\end{equation}
Consequently,
\begin{equation}
T\leq 1+\left[\dfrac{2^{3p}\max\left\{H_{0},(p+\nu-1)(H_{f,p}(\nu)+\theta(p-1)!)\right\}(p+\nu)^{p+\nu-1}\|x_{0}-x^{*}\|^{p+\nu}}{(p-1)!}\right]^{\frac{1}{p+\nu}}\left(\dfrac{1}{\epsilon}\right)^{\frac{1}{p+\nu}}
\label{eq:4.27}
\end{equation}
if $\nu$ is known (i.e., $\alpha=\nu$), and
\begin{equation}
T\leq 1+\left[\dfrac{2^{3p}\max\left\{H_{0},4\theta(p-1)!,(4H_{f,p}(\nu))^{\frac{p}{p+\nu-1}}\left(4D_{0}\right)^{\frac{1-\nu}{p+\nu-1}}\right\}(p+1)^{p}\|x_{0}-x^{*}\|^{p+1}}{(p-1)!}\right]^{\frac{1}{p+1}}\left(\dfrac{1}{\epsilon}\right)^{\frac{p}{(p+1)(p+\nu-1)}}
\label{eq:4.28}
\end{equation}
if $\nu$ is unknown (i.e., $\alpha=1$).
\end{theorem}

\begin{proof}
By Lemma 4.4, we have
\begin{equation*}
2^{i_{t}}H_{t}=2(2^{i_{t}-1}H_{t})=2H_{t+1}\leq 2\max\left\{H_{0},\tilde{N}_{\nu}(\epsilon)\right\},\,\,t=0,\ldots,T-1.
\end{equation*}
Then (\ref{eq:4.26}) follows directly from Theorem 4.3 with
\begin{equation*}
M_{\nu}=2\max\left\{H_{0},\tilde{N}_{\nu}(\epsilon)\right\}.
\end{equation*}
Now, combining (\ref{eq:4.26}) and (\ref{eq:4.23}) for $k=T$, we obtain
\begin{equation*}
\epsilon\leq\dfrac{2^{3p}\max\left\{H_{0},\tilde{N}_{\nu}(\epsilon)\right\}(p+\alpha)^{p+\alpha-1}\|x_{0}-x^{*}\|^{p+\alpha}}{(p-1)!(T-1)^{p+\alpha}}
\end{equation*}
and so,
\begin{equation}
T\leq 1+\left[\dfrac{2^{3p}\max\left\{H_{0},\tilde{N}_{\nu}(\epsilon)\right\}(p+\alpha)^{p+\alpha-1}\|x_{0}-x^{*}\|^{p+\alpha}}{\epsilon(p-1)!}\right]^{\frac{1}{p+\alpha}}.
\label{eq:4.29}
\end{equation}
If $\nu$ is known, then $\alpha=\nu$ and, by (\ref{eq:4.21}), we have
\begin{equation}
\tilde{N}_{\nu}(\epsilon)=(p+\nu-1)(H_{f,p}(\nu)+\theta(p-1)!).
\label{eq:4.30}
\end{equation}
Thus, combining (\ref{eq:4.29}) and (\ref{eq:4.30}), we get (\ref{eq:4.27}). On the other hand, if $\nu$ is unknown, then $\alpha=1$ and, by (\ref{eq:4.21}), (\ref{eq:3.12}) and $\epsilon\in (0,1)$, we have
\begin{eqnarray}
\tilde{N}_{\nu}(\epsilon)&=&\max\left\{4\theta(p-1)!,(4H_{f,p}(\nu))^{\frac{p}{p+\nu-1}}\left(\dfrac{4R(\epsilon)}{\epsilon}\right)^{\frac{1-\nu}{p+\nu-1}}\right\}\nonumber\\
&\leq &\max\left\{4\theta(p-1)!,(4H_{f,p}(\nu))^{\frac{p}{p+\nu-1}}\left(4D_{0}\right)^{\frac{1-\nu}{p+\nu-1}}\right\}\epsilon^{-\frac{1-\nu}{p+\nu-1}}.
\label{eq:4.31}
\end{eqnarray}
In this case, combining (\ref{eq:4.29}) and (\ref{eq:4.31}) we get (\ref{eq:4.28}).
\end{proof}

When $\nu=1$, bounds (\ref{eq:4.27}) and (\ref{eq:4.28}) have the same dependence on $\epsilon$. However, when $\nu\neq 1$, the bound of $\mathcal{O}\left(\epsilon^{-p/(p+1)(p+\nu-1)}\right)$ obtained for the universal scheme (i.e., Algorithm 4 with $\alpha=1$) is worse than the bound of $\mathcal{O}\left(\epsilon^{-1/(p+\nu)}\right)$ obtained for the nonuniversal scheme ($\alpha=\nu$). For high-order methods ($p\geq 2$), to the best of our knowledge, there is no simple procedure by which one can identify the level of smoothness $\nu$ of the $p$th derivatives (in general). Therefore, despite this gap in the complexity bounds, we believe that the automatic choice of the best function subclass in the universal scheme is a very attractive feature. Moreover, in the nonuniversal scheme, for any $\theta>0$ with $(p+\nu-1)(H_{f,p}(\nu)+\theta(p-1)!)>H_{0}$, the corresponding right-hand side of (4.22) has an additional term
\begin{equation*}
\left[\dfrac{2^{3p}\theta (p-1)!(p+\nu)^{p+\nu-1}\|x_{0}-x^{*}\|^{p+\nu}}{(p-1)!}\right]^{\frac{1}{p+\nu}}\left(\frac{1}{\epsilon}\right)^{\frac{1}{p+\nu}}
\end{equation*}
in comparison to its value when $\theta=0$. In contrast, in the accelerated universal scheme, for any $\theta$ in the interval
\begin{equation*}
0<\theta<\min\left\{\dfrac{H_{0}}{4(p-1)!},\dfrac{H_{f,p}(\nu)^{\frac{p}{p+\nu-1}}D_{0}^{\frac{1-\nu}{p+\nu-1}}}{(p-1)!}\right\},
\end{equation*}
the corresponding right-hand side in (4.23) is the same as for $\theta=0$. In this sense, it appears that the accelerated universal scheme is more robust than the accelerated nonuniversal scheme in terms of the inexact solution of the auxiliary problems.

\section{Lower complexity bounds under H\"{o}lder condition}

In this section we investigate how much the convergence rates of our tensor methods can be improved with respect to problems satisfying H1. Specifically, we derive lower complexity bounds for $p$-order tensor methods applied to the problem (\ref{eq:2.1}), where the objective $f$ is convex and $H_{f,p}(\nu)<+\infty$ for some $\nu\in [0,1]$.

\subsection{Hard functions and Lower Complexity Bounds}

For simplicity, let us consider $\E=\mathbb{R}^{n}$ and $B=I_{n}$. Given an approximation $\bar{x}$ for the solution of (\ref{eq:2.1}), $p$-order methods usually compute the next test point as $x^{+}=\bar{x}+\bar{h}$, where the search direction $\bar{h}$ is the solution of an auxiliary problem of the form
\begin{equation}
\min_{h\in\mathbb{R}^{n}}\,\phi_{a,\gamma,m}(h)\equiv \sum_{i=1}^{p}a^{(i)}D^{i}f(\bar{x})[h]^{i}+\gamma\|h\|^{m},
\label{eq:6.1}
\end{equation}
with $a\in\mathbb{R}^{p}$, $\gamma>0$, and $m>1$. Denote by $\Gamma_{\bar{x},f}(a,\gamma,m)$ the set of all stationary points of function $\phi_{a,\gamma,m}(\,.\,)$, and define the linear subspace
\begin{equation}
S_{f}(\bar{x})=\text{Lin}\left(\Gamma_{\bar{x},f}(a,\gamma,m)\,|\,a\in\mathbb{R}^{p},\,\gamma>0,\,m>1\right).
\label{eq:6.2}
\end{equation}
With this notation, we can characterize the class of $p$-order tensor methods by the following assumption.
\\[0.2cm]
\noindent\textbf{Assumption 1.} Given $x_{0}\in\mathbb{R}^{n}$, the method generates a sequence of test points $\left\{x_{k}\right\}_{k\geq 0}$ such that
\begin{equation}
x_{k+1}\in x_{0}+\sum_{i=0}^{k}S_{f}(x_{i}),\quad k\geq 0.
\label{eq:6.3}
\end{equation}

Given $\nu\in [0,1]$, our parametric family of difficult functions for $p$-order tensor methods is defined as
\begin{equation}
f_{k}(x)=\dfrac{1}{p+\nu}\left[\sum_{i=1}^{k-1}|x^{(i)}-x^{(i+1)}|^{p+\nu}+\sum_{i=k}^{n}|x^{(i)}|^{p+\nu}\right]-x^{(1)},\quad 2\leq k\leq n.
\label{eq:6.4}
\end{equation}
The next lemma establishes that for each $f_{k}(\,.\,)$ we have $H_{f_{k},p}(\nu)<+\infty$.

\begin{lemma}
\label{lem:6.1}
Given an integer $k\in [2,n]$, the $p$th derivative of $f_{k}(\,.\,)$ is $\nu$-H\"{o}lder continuous with
\begin{equation}
H_{f_{k},p}(\nu)=2^{\frac{2+\nu}{2}}\Pi_{i=1}^{p-1}(p+\nu-i).
\label{eq:6.5}
\end{equation}
\end{lemma}

\begin{proof}
In view of (\ref{eq:6.4}), we have
\begin{equation}
f_{k}(x)=\eta_{p+\nu}(A_{k}x)-\la e_{1},x\ra,
\label{eq:6.6}
\end{equation}
where
\begin{equation}
\eta_{p+\nu}(u)=\dfrac{1}{p+\nu}\sum_{i=1}^{n}|u^{(i)}|^{p+\nu},
\label{eq:6.7}
\end{equation}
\begin{equation}
A_{k}=\left(\begin{array}{cc} U_{k} & 0\\ 0 & I_{n-k} \end{array}\right),\quad\text{with}\quad U_{k}=\left(\begin{array}{rrrrrr} 1      & -1     &  0     & \ldots & 0      & 0\\
                                  0      &  1     & -1     & \ldots & 0      & 0\\
                                  \vdots & \vdots & \vdots &        & \vdots & \vdots\\
                                  0      &  0     &  0     & \ldots & 1      & -1\\
                                  0      &  0     &  0     & \ldots & 0      &  1
             \end{array}\right)\in\mathbb{R}^{k\times k}.
\label{eq:6.8}
\end{equation}
It can be shown that (see page 13 in \cite{NES6})
\begin{equation}
\|A_{k}\|\leq 2.
\label{eq:6.9}
\end{equation}
On the other hand, for any $x,h\in\mathbb{R}^{n}$, we have
\begin{equation*}
D^{\ell}\eta_{p+\nu}(x)[h]^{\ell}=\left\{\begin{array}{ll}\left(\dfrac{\Pi_{i=0}^{\ell-1}(p+\nu-i)}{p+\nu}\right)\sum_{i=1}^{n}|x^{(i)}|^{p+\nu-\ell}(h^{(i)})^{\ell},&\text{if}\,\,\ell\,\,\text{is even},\\
\left(\dfrac{\Pi_{i=0}^{\ell-1}(p+\nu-i)}{p+\nu}\right)\sum_{i=1}^{n}|x^{(i)}|^{p+\nu-1-\ell}x^{(i)}(h^{(i)})^{\ell},&\text{if}\,\,\ell\,\,\text{is odd}.
                                         \end{array}
                                   \right.
\end{equation*}
Therefore, for all $x,y,h\in\mathbb{R}^{n}$, it follows that
\begin{eqnarray*}
\left|D^{p}\eta_{p+\nu}(x)[h]^{p}-D^{p}\eta_{p+\nu}(y)[h]^{p}\right| & \leq & \left(\Pi_{i=1}^{p-1}(p+\nu-i)\right)\sum_{i=1}^{n}|x^{(i)}-y^{(i)}|^{\nu}(h^{(i)})^{p}\\
&\leq & \left(\Pi_{i=1}^{p-1}(p+\nu-i)\right)\|x-y\|_{\infty}^{\nu}\sum_{i=1}^{n}(h^{(i)})^{p}\\
&\leq &\left(\Pi_{i=1}^{p-1}(p+\nu-i)\right)\|x-y\|_{\infty}^{\nu}\sum_{i=1}^{n}\left[(h^{(i)})^{2}\right]^{\frac{p}{2}}\\
&\leq & \left(\Pi_{i=1}^{p-1}(p+\nu-i)\right)\|x-y\|_{\infty}^{\nu}\|h\|^{p}.
\end{eqnarray*}
Consequently, for all $x,d,h\in\mathbb{R}^{n}$, we have
\begin{eqnarray}
\left|D^{p}f_{k}(x+d)[h]^{p}-D^{p}f_{k}(x)[h]^{p}\right|&=&\left|D^{p}\eta_{p+\nu}(A_{k}(x+d))[A_{k}h]^{p}-D^{p}\eta_{p+\nu}(A_{k}x)[A_{k}h]^{p}\right|\nonumber\\
 &\leq & \Pi_{i=1}^{p-1}(p+\nu-i)\|A_{k}d\|_{\infty}^{\nu}\|A_{k}h\|^{p}.
\label{eq:6.10}
\end{eqnarray}
Note that
\begin{eqnarray}
\|A_{k}d\|_{\infty}&  =  &\max_{1\leq i\leq
n}|(A_{k}d)^{(i)}| \leq \max_{1\leq n-1}
\left(|d^{(i)}|+|d^{(i+1)}|\right)\\
& \leq & \max_{1\leq i\leq
n-1}\sqrt{2[(d^{(i)})^{2}+(d^{(i+1)})^{2}]}
                   \leq  2^{\frac{1}{2}}\|d\|, \nonumber
\label{eq:6.11}
\end{eqnarray}
and, by (\ref{eq:6.9}), that
\begin{equation}
\|A_{k}h\|\leq \|A\|\|h\|\leq 2\|h\|.
\label{eq:6.12}
\end{equation}
Thus, combining (\ref{eq:6.10})-(\ref{eq:6.12}), we get
\begin{equation*}
\|D^{p}f_{k}(x+d)-D^{p}f_{k}(x)\|\leq 2^{\frac{2+\nu}{2}}\Pi_{i=1}^{p-1}(p+\nu-i)\|d\|^{\nu}.
\end{equation*}
\end{proof}

The next lemma provides additional properties of $f_{k}(\,.\,)$.

\begin{lemma}
\label{lem:6.2}
Given an integer $k\in [2,n]$, let function $f_{k}(\,.\,)$ be defined by (\ref{eq:6.4}). Then, $f_{k}(\,.\,)$ has a unique global minimizer $x_{k}^{*}$. Moreover,
\begin{equation}
f_{k}^{*}=-\dfrac{(p+\nu-1)k}{p+\nu}\quad\text{and}\quad \|x_{k}^{*}\|<\dfrac{(k+1)^{\frac{3}{2}}}{\sqrt{3}}.
\label{eq:6.13}
\end{equation}
\end{lemma}

\begin{proof}
The existence and uniqueness of $x_{k}^{*}$ follows from the fact that $f_{k}(\,.\,)$ is uniformly convex. In view of (\ref{eq:6.6}), it follows from the first-order optimality condition that
\begin{equation*}
A_{k}^{T}\nabla\eta_{p+\nu}(A_{k}x_{k}^{*})=e_{1}.
\end{equation*}
Therefore, $A_{k}x_{k}^{*}=y_{k}^{*}$, where $y_{k}^{*}$ satisfies
\begin{equation}
\nabla\eta_{p+\nu}(y_{k}^{*})=A_{k}^{T}e_{1}=\hat{e}_{k}=\left[\begin{array}{c}\bar{e}_{k}\\ 0_{n-k}\end{array}\right]
\label{eq:6.14}
\end{equation}
with $\bar{e}_{k}\in\mathbb{R}^{k}$ being the vector of all ones and $0_{n-k}$ being the origin in $\mathbb{R}^{n-k}$. Note that
\begin{equation}
\dfrac{\partial\eta_{p+\nu}}{\partial y_{i}}(y)=|y^{(i)}|^{p+\nu-2}y^{(i)},\quad i=1,\ldots,n.
\label{eq:6.14extra}
\end{equation}
Consequently, (\ref{eq:6.14}) is equivalent to
\begin{equation*}
|(y_{k}^{*})^{(i)}|^{p+\nu-2}(y_{k}^{*})^{(i)}=\left\{\begin{array}{ll} 1,&\text{for}\,\,i=1,\ldots,k,\\
                                                                        0,&\text{for}\,\,i=k+1,\ldots,n.
                                               \end{array}
                                               \right.
\end{equation*}
Thus,
\begin{equation}
A_{k}x_{k}^{*}=y_{k}^{*}=\hat{e}_{k},
\label{eq:6.15}
\end{equation}
and so
\begin{equation}
(x_{k})^{(i)}=(A_{k}^{-1}y_{k}^{*})^{(i)}=(A_{k}^{-1}\hat{e}_{k})^{(i)}=(k-i+1)_{+},\,\,i=1,\ldots,n,
\label{eq:6.16}
\end{equation}
where $(\tau)_{+}=\max\left\{0,\tau\right\}$. Finally,
combining (\ref{eq:6.6}), (\ref{eq:6.7}), (\ref{eq:6.15})
and (\ref{eq:6.16}) we get
\begin{eqnarray*}
f_{k}^{*}&=&\eta_{p+\nu}(A_{k}x_{k}^{*})-\la e_{1},x_{k}^{*}\ra = \eta_{p+\nu}(\hat{e}_{k})-(x_{k}^{*})^{(1)}\\
&=&\dfrac{1}{p+\nu}\sum_{i=1}^{n}|(\hat{e}_{k})^{(i)}|^{p+\nu}-k=\dfrac{k}{p+\nu}-k
=-\dfrac{(p+\nu-1)k}{p+\nu},
\end{eqnarray*}
\begin{eqnarray*}
\|x_{k}^{*}\|^{2}&=&\sum_{i=1}^{n}\left[(x_{k}^{*})^{(i)}\right]^{2}=k^{2}+(k-1)^{2}+\ldots +2^{2}+1^{2}\\
                 &=&\sum_{i=1}^{k}i^{2}=\dfrac{k(k+1)(2k+1)}{6}<\dfrac{(k+1)^{3}}{3}.
\end{eqnarray*}
\end{proof}

Our goal is to understand the behavior of the tensor methods specified by Assumption 1 when applied to the minimization of $f_{k}(\,.\,)$ with a suitable $k$. For that, let us consider the following subspaces:
\begin{equation*}
\mathbb{R}^{n}_{k}=\left\{x\in\mathbb{R}^{n}\,|\,x^{(i)}=0,\,\,i=k+1,\ldots,n\right\},\,\,1\leq k\leq n-1.
\end{equation*}

\begin{lemma}
\label{lem:6.3}
For any $q\geq 0$ and $x\in\mathbb{R}_{k}^{n}$, $f_{k+q}(x)=f_{k}(x)$.
\end{lemma}

\begin{proof}
It follows directly from (\ref{eq:6.4}).
\end{proof}

\begin{lemma}
\label{lem:6.4}
Let $\mathcal{M}$ be a $p$-order tensor method satisfying Assumption 1. If $\mathcal{M}$ is applied to the minimization of $f_{t}(\,.\,)$ starting from $x_{0}=0$, then the sequence $\left\{x_{k}\right\}_{k\geq 0}$ of test points generated by $\mathcal{M}$ satisfies
\begin{equation*}
x_{k+1}\in\sum_{i=0}^{k}S_{f_{t}}(x_{i})\subset\mathbb{R}^{n}_{k+1},\,\,0\leq k \leq t-1.
\end{equation*}
\end{lemma}

\begin{proof}
See Lemma 2 in \cite{NES6}.
\end{proof}

Now, we can prove the lower complexity bound for $p$-order tensor methods applied to the minimization of functions with $\nu$-H\"{o}lder continuous $p$th derivatives.

\begin{theorem}
\label{thm:6.1}
Let $\mathcal{M}$ be a $p$-order tensor method satisfying Assumption 1. Assume that for any function $f$ with $H_{f,p}(\nu)<+\infty$ this method ensures the rate of convergence:
\begin{equation}
\min_{0\leq k\leq t}\,f(x_{k})-f^{*}\leq\dfrac{H_{f,p}(\nu)\|x_{0}-x^{*}\|^{p+\nu}}{\kappa(t)},\,\,t\geq 1,
\label{eq:6.17}
\end{equation}
where $\left\{x_{k}\right\}_{k\geq 0}$ is the sequence generated by method $\mathcal{M}$ and $x^{*}$ is a global minimizer of $f$. Then, for all $t\geq 1$ such that $2t+1\leq n$ we have
\begin{equation}
\kappa(t)\leq C_{p,\nu}(t+1)^{\frac{3(p+\nu)-2}{2}},
\label{eq:6.18}
\end{equation}
where
\begin{equation}
C_{p,\nu}=\dfrac{2^{\frac{3p+4\nu+2}{2}}\Pi_{i=0}^{p-1}(p+\nu-i)}{3^{\frac{p+\nu}{2}}(p+\nu-1)}.
\label{eq:6.19}
\end{equation}
\end{theorem}

\begin{proof}
Let us apply method $\mathcal{M}$ for minimizing function $f_{2t+1}(\,.\,)$ starting from point $x_{0}=0$. By Lemma \ref{lem:6.4} we have $x_{i}\in\mathbb{R}_{t}^{n}$ for all $i$, $0\leq i\leq t$. Moreover, by Lemma \ref{lem:6.3} we have
\begin{equation}
f_{2t+1}(x)=f_{t}(x),\quad\forall x\in\mathbb{R}^{n}_{t}.
\label{eq:6.20}
\end{equation}
Thus, from (\ref{eq:6.17}), (\ref{eq:6.20}), Lemma \ref{lem:6.1} and Lemma \ref{lem:6.2} we get
\begin{eqnarray*}
\kappa(t)&\leq &\dfrac{H_{f_{2t+1},p}(\nu)\|x_{0}-x_{2t+1}^{*}\|^{p+\nu}}{\min_{0\leq k\leq t}\,f_{2t+1}(x_{k})-f_{2t+1}^{*}}=\dfrac{2^{\frac{2+\nu}{2}}\Pi_{i=1}^{p-1}(p+\nu-i)\|x_{2t+1}^{*}\|^{p+\nu}}{\min_{0\leq k\leq t}\,f_{t}(x_{k})-f_{2t+1}^{*}}\\
&\leq & \dfrac{2^{\frac{2+\nu}{2}}\Pi_{i=1}^{p-1}(p+\nu-i)\|x_{2t+1}^{*}\|^{p+\nu}}{f_{t}^{*}-f_{2t+1}^{*}}<\dfrac{2^{\frac{2+\nu}{2}}\Pi_{i=1}^{p-1}(p+\nu-i)\left[2(t+1)\right]^{\frac{3}{2}(p+\nu)}}{3^{\frac{p+\nu}{2}}(f_{t}^{*}-f_{2t+1}^{*})}\\
&=& \dfrac{2^{\frac{3p+4\nu+2}{2}}\Pi_{i=1}^{p-1}(p+\nu-i)(t+1)^{\frac{3(p+\nu)}{2}}}{3^{\frac{p+\nu}{2}}\left[\dfrac{-(p+\nu-1)t+(p+\nu-1)(2t+1)}{p+\nu}\right]}\\
&=&
\dfrac{2^{\frac{3p+4\nu+2}{2}}\Pi_{i=0}^{p-1}(p+\nu-i)(t+1)^{\frac{3(p+\nu)}{2}}}
{3^{\frac{p+\nu}{2}}(p+\nu-1)(t+1)} =
C_{p,\nu}(t+1)^{\frac{3(p+\nu)-2}{2}},
\end{eqnarray*}
where constant $C_{p,\nu}$ is given by (\ref{eq:6.19}).
\end{proof}

\subsection{Discussion}

Theorem \ref{thm:6.1} establishes that the lower bound for
the rate of convergence of tensor methods applied to
functions with $\nu$-H\"{o}lder continuous $p$th
derivatives is of
$\mathcal{O}\left((\frac{1}{k})^{\frac{3(p+\nu)-2}{2}}\right)$.
In the Lipschitz case (i.e., $\nu=1$) we have
$\mathcal{O}\left((\frac{1}{k})^{\frac{3p+1}{2}}\right)$,
which coincide with the bounds in \cite{AR,NES6}. On the
other hand, for first-order methods (i.e., $p=1$) we have
$\mathcal{O}\left((\frac{1}{k})^{\frac{1+3\nu}{2}}\right)$,
which is the bound in \cite{NEM}.

The rate of
$\mathcal{O}\left((\frac{1}{k})^{\frac{3(p+\nu)-2}{2}}\right)$
corresponds to a worst-case complexity bound of
\\$\mathcal{O}\left(\epsilon^{-2/[3(p+\nu)-2]}\right)$
iterations necessary to ensure
$
f(x_{k})-f(x^{*})\leq\epsilon.
$
This means that the nonuniversal accelerated schemes proposed in this paper (e.g., Algorithm 4 with $\alpha=\nu$) are nearly optimal tensor methods. In fact, for $\epsilon\in (0,1)$, note that
\begin{eqnarray*}
\left(\frac{1}{\epsilon}\right)^{\frac{1}{p+\nu}}&=&\left(\frac{1}{\epsilon}\right)^{\frac{p+\nu-2}{(p+\nu)[3(p+\nu)-2]}}\left(\frac{1}{\epsilon}\right)^{\frac{2}{3(p+\nu)-2}}\leq \left(\frac{1}{\epsilon}\right)^{\frac{p-1}{3p^{2}-2p}}\left(\frac{1}{\epsilon}\right)^{\frac{2}{3(p+\nu)-2}}\\
&\leq & \left(\frac{1}{\epsilon}\right)^{\frac{1}{8}}\left(\frac{1}{\epsilon}\right)^{\frac{2}{3(p+\nu)-2}}\\
\end{eqnarray*}
In particular, if $\epsilon=10^{-6}$,  we have
$
\left(\frac{1}{\epsilon}\right)^{\frac{1}{p+\nu}}\leq 6\left(\frac{1}{\epsilon}\right)^{\frac{2}{3(p+\nu)-2}}.
$ Thus, in practice, the complexity bounds of our
accelerated nonuniversal methods differ from the lower
bound just by a small constant factor.

Notice that the lower-bound described in Theorem 5.5 is only valid while the iteration counter $t$ satisfies $t<\frac{1}{2}(n-1)$, where $n$ is the dimension of the domain of the objective. The same condition on $t$ appears in other lower bounds in the literature for the case $p=1$ and $\nu=1$ (see, e.g., Theorem 2.1.7 in \cite{NES18}).

\section{Conclusion}

In this paper, we presented $p$-order methods for
unconstrained minimization of convex functions that are
$p$-times differentiable with $\nu$-H\"{o}lder continuous
$p$th derivatives. For the universal and the nonuniversal
schemes without acceleration, we established iteration
complexity bounds of
$\mathcal{O}\left(\epsilon^{-1/(p+\nu-1)}\right)$ for
reducing the functional residual below a given
$\epsilon\in (0,1)$. Assuming that $\nu$ is known, we
obtained an improved complexity bound of
$\mathcal{O}\left(\epsilon^{-1/(p+\nu)}\right)$ for the corresponding
accelerated scheme. For the case in which $\nu$ is
unknown, we presented an accelerated universal tensor scheme
with an iteration complexity of
$\mathcal{O}\left(\epsilon^{-p/[(p+1)(p+\nu-1)]}\right)$. 

Finally, a lower complexity bound of
$\mathcal{O}(\epsilon^{-2/[3(p+\nu)-2]})$ was also obtained for
the referred problem class. This means that, in practice, our
accelerated nonuniversal schemes are nearly optimal.
Remarkably, the complexity bound obtained for the
accelerated universal schemes is slightly worse than the
bound obtained for the nonuniversal accelerated schemes.
Up to now, it is not clear whether the estimating
sequences technique can be modified to provide an
accelerated universal $p$-order method with a complexity
bound of $\mathcal{O}\left(\epsilon^{-1/(p+\nu)}\right)$.

It is worth mentioning that the study of high-order methods is still at its early stages, with the majority of recent works in this area focusing on the derivation of global complexity bounds (see, e.g., \cite{BAES,Birgin,CGT2,CHEN,Mart,NES6}). These bounds predict that high-order methods with $p\geq 3$ may require significantly fewer calls of the oracle than second-order methods. As pointed out in \cite{CHEN,NES6}, the computation of high-order derivatives may be affordable for structured objectives (such as separable functions). Moreover, at least for $p=3$ and $\alpha=\nu$, the auxiliary problems can be solved using Bregman gradient methods that also take into account their particular structure \cite{GN3,NES6}. Nevertheless, the practical impact of high-order methods is yet to be seen.

\appendix

\section{Auxiliary Results}

In all algorithms described in this paper, the acceptance of new points is conditioned to the achievement of a sufficient decrease of the objective function value. In the nonaccelerated schemes, the sufficient decrease condition is specified by (3.4), while for accelerated schemes, it is specified by (4.6). In this Appendix we present auxiliary results from which we conclude that these conditions are satisfied when the regularization parameter is sufficiently large.

\subsection{Results for schemes without acceleration}
Our first lemma gives a lower bound for the functional decrease in terms of a suitable power of the norm of the displacement, when $\nu$ is known.


\begin{lemma}
\label{lem:extraA.1}
Let $H_{f,p}(\nu)<+\infty$ for some $\nu\in [0,1]$ and assume that $x^{+}$ satisfies
\begin{equation}
\Omega_{\bar{x},p,H}^{(\nu)}(x^{+})\leq f(\bar{x}),
\label{eq:extraA.1}
\end{equation}
for some $\bar{x}\in\E$ and $H>0$. If $H\geq\frac{3}{2}H_{f,p}(\nu)$, then
\begin{equation}
f(\bar{x})-f(x^{+})\geq\dfrac{H}{(p+1)!}\|x^{+}-\bar{x}\|^{p+\nu}.
\label{eq:extraA.2}
\end{equation}
\end{lemma}

\begin{proof}
In view of (\ref{eq:2.8}) and (\ref{eq:extraA.1}), we have
\begin{eqnarray*}
f(x^{+})&\leq & \Omega_{\bar{x},p,H_{f,p}(\nu)}^{(\nu)}(x^{+})\\
        & =   & \Phi_{\bar{x},p}(x^{+})+\dfrac{H}{p!}\|x^{+}-\bar{x}\|^{p+\nu}-\dfrac{(H-H_{f,p}(\nu))}{p!}\|x^{+}-\bar{x}\|^{p+\nu}\\
        & =   &\Omega_{\bar{x},p,H}^{(\nu)}(x^{+})-\dfrac{(H-H_{f,p}(\nu))}{p!}\|x^{+}-\bar{x}\|^{p+\nu}\\
        &\leq & f(\bar{x})-\dfrac{(H-H_{f,p}(\nu))}{p!}\|x^{+}-\bar{x}\|^{p+\nu},
\end{eqnarray*}
which gives
\begin{equation*}
f(\bar{x})-f(x^{+})\geq \dfrac{H-H_{f,p}(\nu)}{p!}\|x^{+}-\bar{x}\|^{p+\nu}
\end{equation*}
Since $H\geq\frac{3}{2}H_{f,p}(\nu)\geq\dfrac{p+1}{p}H_{f,p}(\nu)$ for all $p\geq 2$, it follows that
\begin{equation*}
f(\bar{x})-f(x^{+})\geq\dfrac{\left(1-\frac{p}{p+1}\right)H}{p!}\|x^{+}-\bar{x}\|^{p+\nu}=\dfrac{H}{(p+1)!}\|x^{+}-\bar{x}\|^{p+\nu}.
\end{equation*}
\end{proof}

The next lemma provides a lower bound for the functional decrease in terms of a suitable power of the norm of the gradient when $\nu$ is known.
\begin{lemma}
\label{lem:extraA.2}
Let $H_{f,p}(\nu)<+\infty$ for some $\nu\in [0,1]$, and assume that $x^{+}\in\E$ satisfies (\ref{eq:extraA.1}) and
\begin{equation}
\|\nabla\Omega_{\bar{x},p,H}^{(\nu)}(x^{+})\|_{*}\leq\theta\|x^{+}-\bar{x}\|^{p+\nu-1}
\label{eq:extraA.3}
\end{equation}
for some $\bar{x}\in\E$, $H>0$, and $\theta\geq 0$. If
\begin{equation}
H\geq\max\left\{\dfrac{3H_{f,p}(\nu)}{2},3\theta (p-1)!\right\},
\label{eq:extraA.4}
\end{equation}
then
\begin{equation}
f(\bar{x})-f(x^{+})\geq\dfrac{1}{8(p+1)!H^{\frac{1}{p+\nu-1}}}\|\nabla f(x^{+})\|_{*}^{\frac{p+\nu}{p+\nu-1}}.
\label{eq:extraA.5}
\end{equation}
\end{lemma}

\begin{proof}
By (\ref{eq:2.5}), (\ref{eq:2.7}), (\ref{eq:extraA.3}), and (\ref{eq:extraA.4}), we have
\begin{eqnarray*}
\|\nabla f(x^{+})\|_{*}&\leq &\|\nabla f(x^{+})-\nabla\Phi_{\bar{x},p}(x^{+})\|_{*}+\|\nabla\Phi_{\bar{x},p}(x^{+})-\nabla\Omega_{\bar{x},p,H}^{(\nu)}(x^{+})\|_{*}\\
                       &     & +\|\nabla\Omega_{\bar{x},p,H}^{(\nu)}(x^{+})\|_{*}\\
                       &\leq & \left[\dfrac{H_{f,p}(\nu)}{(p-1)!}+\dfrac{H(p+\nu)}{p!}+\theta \right]\|x^{+}-\bar{x}\|^{p+\nu-1}\\
                       &\leq & 2H\|x^{+}-\bar{x}\|^{p+\nu-1}.
\end{eqnarray*}
Thus,
\begin{equation}
\|x^{+}-\bar{x}\|^{p+\nu}\geq\left(\dfrac{1}{2H}\right)^{\frac{p+\nu}{p+\nu-1}}\|\nabla f(x^{+})\|_{*}^{\frac{p+\nu}{p+\nu-1}}.
\label{eq:extraA.6}
\end{equation}
On the other hand, by (\ref{eq:extraA.1}) and (\ref{eq:extraA.4}) it follows from Lemma \ref{lem:extraA.1} that
\begin{equation}
f(\bar{x})-f(x^{+})\geq\dfrac{H}{(p+1)!}\|x^{+}-\bar{x}\|^{p+\nu}.
\label{eq:extraA.7}
\end{equation}
Then, combining (\ref{eq:extraA.6}), and (\ref{eq:extraA.7}) we get (\ref{eq:extraA.5}).
\end{proof}

The lemma below gives lower bounds for powers of the norm of the displacement when $\nu$ is unknown.
\begin{lemma}
\label{lem:A.3}
Let $H_{f,p}(\nu)<+\infty$ for some $\nu\in [0,1]$, and assume that $x^{+}$ satisfies
\begin{equation}
\|\nabla\Omega_{\bar{x},p,H}^{(1)}(x^{+})\|_{*}\leq\theta\|x^{+}-\bar{x}\|^{p}
\label{eq:A.4}
\end{equation}
for some $\bar{x}\in\E$, $H>0$ and $\theta\geq 0$. If for some $\delta>0$ we have
\begin{equation}
\|\nabla f(x^{+})\|_{*}\geq\delta\quad\text{and}\quad H\geq\max\left\{\theta,(CH_{f,p}(\nu))^{\frac{p}{p+\nu-1}}\left(\dfrac{4}{\delta}\right)^{\frac{1-\nu}{p+\nu-1}}\right\},
\label{eq:A.5}
\end{equation}
with constant $C\geq 1$, then
\begin{equation}
\|x^{+}-\bar{x}\|^{1-\nu}\geq\dfrac{CH_{f,p}(\nu)}{H}
\label{eq:A.6}
\end{equation}
and, consequently,
\begin{equation}
4H\|x^{+}-\bar{x}\|^{p}\geq\|\nabla f(x^{+})\|_{*}.
\label{eq:A.7}
\end{equation}
\end{lemma}
\begin{proof}
For $\nu=1$, (\ref{eq:A.6}) is obvious. Thus, assume that $\nu\in [0,1)$ and denote $r=\|x^{+}-\bar{x}\|$. Then, by (\ref{eq:2.5}), (\ref{eq:2.7}), and (\ref{eq:A.4}), we have
\begin{eqnarray}
\delta &<& \|\nabla f(x^{+})\|_{*}\nonumber\\
       &\leq & \|\nabla f(x^{+})-\nabla\Phi_{\bar{x},p}(x^{+})\|_{*}
       +\|\nabla\Phi_{\bar{x},p}(x^{+})-\nabla\Omega_{\bar{x},p,H}^{(1)}(x^{+})\|_{*}\nonumber\\
       & & +\|\nabla\Omega_{\bar{x},p,H}^{(1)}(x^{+})\|_{*}\nonumber\\
       &\leq &\dfrac{H_{f,p}(\nu)}{(p-1)!}r^{p+\nu-1}+\left(\dfrac{H(p+1)}{p!}+\theta\right)r^{p}\nonumber\\
       & =   &r^{p+\nu-1}\left[\dfrac{H_{f,p}(\nu)}{(p-1)!}+\left(\dfrac{p+1}{p!}+\dfrac{\theta}{H}\right)Hr^{1-\nu}\right].
\label{eq:A.8}
\end{eqnarray}
Assume by contradiction that (\ref{eq:A.6}) is not true, i.e., $Hr^{1-\nu}<CH_{f,p}(\nu)$. Since $H\geq\theta$ and $C\geq 1$, it follows that
\begin{eqnarray*}
\delta &<& r^{p+\nu-1}\left[\dfrac{H_{f,p}(\nu)}{(p-1)!}+\left(\dfrac{p+1}{p!}+1\right)CH_{f,p}(\nu)\right]\\
       & = &\dfrac{r^{p+\nu-1}H_{f,p}(\nu)}{(p-1)!}\left[1+\dfrac{C(p+1)}{p}+C(p-1)!\right]\\
       &\leq & 4CH_{f,p}(\nu)r^{p+\nu-1}
       < 4CH_{f,p}(\nu)\left(\dfrac{CH_{f,p}(\nu)}{H}\right)^{\frac{p+\nu-1}{1-\nu}}\\
       & =   & 4(CH_{f,p}(\nu))^{\frac{p}{1-\nu}}\left(\dfrac{1}{H}\right)^{\frac{p+\nu-1}{1-\nu}}.
\end{eqnarray*}
This implies that
$
H<(CH_{f,p}(\nu))^{\frac{p}{p+\nu-1}}\left(\dfrac{4}{\delta}\right)^{\frac{1-\nu}{p+\nu-1}}
$
contradicting the second inequality in (\ref{eq:A.5}). Therefore, (\ref{eq:A.6}) holds.

Finally, let us prove (\ref{eq:A.7}). In view of inequality (\ref{eq:A.6}), we have
\begin{equation*}
\dfrac{H_{f,p}(\nu)}{(p-1)!}\leq \dfrac{H}{C(p-1)!}r^{1-\nu}.
\end{equation*}
Thus, it follows from (\ref{eq:A.8}) that
\begin{eqnarray*}
\|\nabla f(x^{+})\|_{*}&\leq & r^{p+\nu-1}\left[\dfrac{H}{C(p-1)!}r^{1-\nu}+\left(\dfrac{p+1}{p!}+\dfrac{\theta}{H}\right)Hr^{1-\nu}\right]\\
                       & =  &r^{p}H\left[\dfrac{1}{Cp!}+\dfrac{(p+1)}{p!}+\dfrac{\theta}{H}\right]\leq  4r^{p}H.
\end{eqnarray*}
\end{proof}

Now, using Lemma \ref{lem:A.3}, we obtain a lower bound for the functional decrease in terms of a computable power of the norm of the displacement, when $\nu$ is unknown.

\begin{lemma}
\label{lem:A.4}
Let $H_{f,p}(\nu)<+\infty$ for some $\nu\in [0,1]$ and assume that $x^{+}\in\E$ satisfies
\begin{equation}
\Omega_{\bar{x},p,H}^{(1)}(x^{+})<f(\bar{x})
\label{eq:A.9}
\end{equation}
and
\begin{equation}
\|\nabla\Omega_{\bar{x},p,H}^{(1)}(x^{+})\|_{*}\leq\theta\|x^{+}-\bar{x}\|^{p}
\label{eq:A.10}
\end{equation}
for some $\bar{x}\in\E$, $H>0$ and $\theta\geq 0$. If for some $\delta>0$ we have
\begin{equation}
\|\nabla f(x^{+})\|_{*}\geq\delta\quad\text{and}\quad H\geq\max\left\{\theta,(CH_{f,p}(\nu))^{\frac{p}{p+\nu-1}}\left(\dfrac{4}{\delta}\right)^{\frac{1-\nu}{p+\nu-1}}\right\},
\label{eq:A.11}
\end{equation}
with constant $C\geq\frac{3}{2}$, then
\begin{equation}
f(\bar{x})-f(x^{+})\geq\dfrac{H}{(p+1)!}\|x^{+}-\bar{x}\|^{p+1}.
\label{eq:A.12}
\end{equation}
\end{lemma}

\begin{proof}
In view of (\ref{eq:2.4}), (\ref{eq:2.7}), and (\ref{eq:A.9}), we have
\begin{eqnarray*}
f(x^{+})&\leq & \Omega_{\bar{x},p,H_{f,p}(\nu)}^{(\nu)}(x^{+})\\
        & =   & \Phi_{\bar{x},p}(x^{+})+\dfrac{H_{f,p}(\nu)}{p!}\|x^{+}-\bar{x}\|^{p+\nu}\\
        & =   &\Phi_{\bar{x},p}(x^{+})+\dfrac{H}{p!}\|x^{+}-\bar{x}\|^{p+1}-\dfrac{H}{p!}\|x^{+}-\bar{x}\|^{p+1}+\dfrac{H_{f,p}(\nu)}{p!}\|x^{+}-\bar{x}\|^{p+\nu}\\
        & =  &\Omega_{\bar{x},p,H}^{(1)}(x^{+})-\dfrac{H}{p!}\|x^{+}-\bar{x}\|^{p+1}+\dfrac{H_{f,p}(\nu)}{p!}\|x^{+}-\bar{x}\|^{p+\nu}\\
        & <  & f(\bar{x})-\dfrac{H}{p!}\|x^{+}-\bar{x}\|^{p+1}+\dfrac{H_{f,p}(\nu)}{p!}\|x^{+}-\bar{x}\|^{p+\nu}
\end{eqnarray*}
and so
\begin{equation}
f(\bar{x})-f(x^{+})\geq \dfrac{H}{p!}\|x^{+}-\bar{x}\|^{p+1}-\dfrac{H_{f,p}(\nu)}{p!}\|x^{+}-\bar{x}\|^{p+\nu}.
\label{eq:A.13}
\end{equation}
Assume by contradiction that (\ref{eq:A.12}) is not true, i.e.,
\begin{equation}
f(\bar{x})-f(x^{+})<\dfrac{H}{(p+1)!}\|x^{+}-\bar{x}\|^{p+1}.
\label{eq:A.14}
\end{equation}
Then, combining (\ref{eq:A.13}) and (\ref{eq:A.14}), we obtain
\begin{equation*}
\dfrac{H}{p!}\|x^{+}-\bar{x}\|^{p+1}-\dfrac{H_{f,p}(\nu)}{p!}\|x^{+}-\bar{x}\|^{p+\nu}<\dfrac{H}{(p+1)!}\|x^{+}-\bar{x}\|^{p+1}
\end{equation*}
which implies that
\begin{equation}
H\left(1-\dfrac{1}{p+1}\right)<H_{f,p}(\nu)\|x^{+}-\bar{x}\|^{\nu-1}.
\label{eq:A.15}
\end{equation}
By (\ref{eq:A.10}) and (\ref{eq:A.11}), the conclusions of Lemma \ref{lem:A.3} hold. In particular, we have
\begin{equation*}
4H\|x^{+}-\bar{x}\|^{p}\geq\|\nabla f(x^{+})\|_{*}\geq\delta
\end{equation*}
and so
\begin{equation}
\|x^{+}-\bar{x}\|^{\nu-1}\leq\left(\dfrac{\delta}{4H}\right)^{\frac{\nu-1}{p}}.
\label{eq:A.16}
\end{equation}
Then it follows from (\ref{eq:A.15}) and (\ref{eq:A.16}) that
\begin{equation*}
\dfrac{Hp}{p+1}<H_{f,p}(\nu)\left(\dfrac{\delta}{4H}\right)^{\frac{\nu-1}{p}}\Longrightarrow H<\left(\frac{3}{2}H_{f,p}(\nu)\right)^{\frac{p}{p+\nu-1}}\left(\dfrac{4}{\delta}\right)^{\frac{1-\nu}{p+\nu-1}},
%
\end{equation*}
contradicting the second inequality in (\ref{eq:A.11}). Therefore, (\ref{eq:A.12}) is true.
\end{proof}

Finally, the next lemma gives a lower bound for the functional decrease in terms of a computable power of the norm of the gradient when $\nu$ is unknown.

\begin{corollary}
\label{cor:A.5}
Let $H_{f,p}(\nu)<+\infty$ for some $\nu\in [0,1]$, and assume that $x^{+}\in\E$ satisfies (\ref{eq:A.9}) and (\ref{eq:A.10}) for some $\bar{x}\in\E$, $H>0$, and $\theta\geq 0$. Given $\delta>0$, define
\begin{equation}
\xi_{\nu}(\delta)\equiv\max\left\{\theta,\left(\dfrac{3H_{f,p}(\nu)}{2}\right)^{\frac{p}{p+\nu-1}}\left(\dfrac{4}{\delta}\right)^{\frac{1-\nu}{p+\nu-1}}\right\}.
\label{eq:A.17}
\end{equation}
If $\|\nabla f(x^{+})\|_{*}\geq\delta$ and $H\geq\xi_{\nu}(\delta)$, then
\begin{equation*}
f(\bar{x})-f(x^{+})\geq\dfrac{1}{8(p+1)!H^{\frac{1}{p}}}\|\nabla f(x^{+})\|^{\frac{p+1}{p}}.
\end{equation*}
\end{corollary}

\begin{proof}
From inequality (\ref{eq:A.7}) in Lemma \ref{lem:A.3} we have
\begin{equation*}
\|x^{+}-\bar{x}\|^{p}\geq\dfrac{1}{4H}\|\nabla f(x^{+})\|_{*}
\end{equation*}
which implies that
\begin{equation*}
\|x^{+}-\bar{x}\|^{p+1}\geq\left(\dfrac{1}{4H}\right)^{\frac{p+1}{p}}\|\nabla f(x^{+})\|_{*}^{\frac{p+1}{p}}.
\end{equation*}
Then, it follows from inequality (\ref{eq:A.12}) in Lemma \ref{lem:A.4} that
\begin{eqnarray*}
f(\bar{x})-f(x^{+})& \geq &\dfrac{H}{(p+1)!}\|x^{+}-\bar{x}\|^{p+1}\geq \dfrac{H}{(p+1)!}\dfrac{1}{4^{\frac{p+1}{p}}H^{\frac{p+1}{p}}}\|\nabla f(x^{+})\|_{*}^{\frac{p+1}{p}}\\
                   &\geq &\dfrac{1}{8(p+1)!H^{\frac{1}{p}}}\|\nabla f(x^{+})\|_{*}^{\frac{p+1}{p}}.
\end{eqnarray*}
\end{proof}

%

\subsection{Results for accelerated schemes}
For the case in which $\nu$ is known, the lemma below establishes that (4.6) is achievable when the regularization parameter is sufficiently large.
\begin{lemma}
\label{lem:A.6}
Let $H_{f,p}(\nu)<+\infty$ for some $\nu\in [0,1]$, and assume that $x^{+}$ satisfies
\begin{equation}
\|\nabla\Omega_{\bar{x},p,H}^{(\nu)}(x^{+})\|_{*}\leq\theta\|x^{+}-\bar{x}\|^{p+\nu-1}
\label{eq:A.18}
\end{equation}
for some $\bar{x}\in\E$, $H>0$ and $\theta\geq 0$. If
\begin{equation}
H\geq (p+\nu-1)\left(H_{f,p}(\nu)+\theta (p-1)!\right),
\label{eq:A.19}
\end{equation}
then
\begin{equation}
\la\nabla f(x^{+}),\bar{x}-x^{+} \ra\geq\dfrac{1}{3}\left[\dfrac{(p-1)!}{H}\right]^{\frac{1}{p+\nu-1}}\|\nabla f(x^{+})\|_{*}^{\frac{p+\nu}{p+\nu-1}}.
\label{eq:A.20}
\end{equation}
\end{lemma}

\begin{proof}
Denote $r=\|x^{+}-\bar{x}\|$. Then, by (\ref{eq:2.5}), (\ref{eq:2.7}), and (\ref{eq:A.18}), we have
\begin{eqnarray*}
\|\nabla f(x^{+})+\dfrac{H(p+\nu)}{p!}r^{p+\nu-2}B(x^{+}-\bar{x})\|_{*}&=&\|\nabla f(x^{+})-\nabla\Phi_{\bar{x},p}(x^{+})+\nabla\Omega_{\bar{x},p,H}^{(\nu)}(x^{+})\|_{*}\\
&\leq & \|\nabla f(x^{+})-\nabla\Phi_{\bar{x},p}(x^{+})\|_{*}+\|\nabla\Omega_{\bar{x},p,H}^{(\nu)}(x^{+})\|_{*}\\
& \leq &\left(\dfrac{H_{f,p}(\nu)}{(p-1)!}+\theta\right)r^{p+\nu-1}.
\end{eqnarray*}
Thus, we obtain
\begin{eqnarray*}
\left(\dfrac{H_{f,p}(\nu)}{(p-1)!}+\theta\right)r^{2(p+\nu-1)}&\geq &\|\nabla f(x^{+})+\dfrac{H(p+\nu)}{p!}r^{p+\nu-2}B(x^{+}-\bar{x})\|_{*}^{2}\\
& = &\|\nabla f(x^{+})\|_{*}^{2}+\dfrac{2(p+\nu)}{p!}Hr^{p+\nu-2}\langle\nabla f(x^{+}),x^{+}-\bar{x}\rangle\\
& &+\dfrac{H^{2}(p+\nu)^{2}}{(p!)^{2}}r^{2(p+\nu-1)}
\end{eqnarray*}
which implies that
\begin{equation}
\la \nabla f(x^{+}),\bar{x}-x^{+}\ra \geq\dfrac{p!}{2(p+\nu)Hr^{p+\nu-2}}\|\nabla f(x^{+})\|_{*}^{2}+\dfrac{p^{2}H}{2(p+\nu)p!}\left[1-\left(\dfrac{H_{f,p}(\nu)+\theta (p-1)!}{H}\right)^{2}\right]r^{p+\nu}.
\label{eq:A.21}
\end{equation}
For $\nu=0$, (\ref{eq:A.21}) leads to the desired relation. Let us assume that $\nu>0$. Denote $g=\|\nabla f(x^{+})\|$ and
$
\Delta^{2}=1-\left(\dfrac{H_{f,p}(\nu)+\theta (p-1)!}{H}\right)^{2}.
$
By (\ref{eq:A.19}), we have
\begin{equation}
\Delta^{2}\geq 1-\dfrac{1}{(p+\nu-1)^{2}}=\dfrac{(p+\nu-1)^{2}-1}{(p+\nu-1)^{2}}=\dfrac{(p+\nu-2)(p+\nu)}{(p+\nu-1)^{2}}>0.
\label{eq:A.22}
\end{equation}
Consider the right-hand side of inequality (\ref{eq:A.21}) as a function of $r$:
\begin{equation*}
h(r)=\dfrac{p!}{2(p+\nu)Hr^{p+\nu-2}}g^{2}+\dfrac{Hp^{2}\Delta^{2}r^{p+\nu}}{2(p+\nu)p!}.
\end{equation*}
Since $\Delta^{2}>0$, $h$ is a convex function for $r>0$. Thus, let us find the optimal $r_{*}$ as a solution to the first-order optimality condition for function $h$:
\begin{equation*}
\dfrac{g^{2}(p+\nu-2)p!}{(p+\nu)Hr_{*}^{p+\nu-1}}=\dfrac{Hp^{2}\Delta^{2}r_{*}^{p+\nu-1}}{p!}.
\end{equation*}
Solving this equation for $r_{*}$, we obtain
$
r_{*}^{p+\nu-1}=\dfrac{g(p-1)!}{H\Delta}\sqrt{\dfrac{p+\nu-2}{p+\nu}}.
$
Consequently,
\begin{eqnarray*}
h(r_{*})& = &\dfrac{r_{*}}{2H(p+\nu)}\left[\dfrac{g^{2}p!}{r^{p+\nu-1}}+\dfrac{H^{2}p^{2}\Delta^{2}r_{*}^{p+\nu-1}}{p!}\right]\\
& = & \dfrac{(p+\nu-1)p\Delta^{\frac{p+\nu-2}{p+\nu-1}}}{(p+\nu)\sqrt{(p+\nu-2)(p+\nu)}}\left(\sqrt{\dfrac{p+\nu-2}{p+\nu}}\right)^{\frac{1}{p+\nu-1}}\left[\dfrac{(p-1)!}{H}\right]^{\frac{1}{p+\nu-1}}g^{\frac{p+\nu}{p+\nu-1}}
\end{eqnarray*}
Now, usinig (\ref{eq:A.22}) we obtain
\begin{eqnarray*}
h(r_{*})&\geq & \frac{(p+\nu-1)p}{(p+\nu)\sqrt{(p+\nu-2)(p+\nu)}}\left(\frac{\sqrt{(p+\nu-2)(p+\nu)}}{p+\nu-1}\right)^{\frac{p+\nu-2}{p+\nu-1}}\left(\sqrt{\frac{p+\nu-2}{p+\nu}}\right)^{\frac{1}{p+\nu-1}}\left[\dfrac{(p-1)!}{H}\right]^{\frac{1}{p+\nu-1}}g^{\frac{p+\nu}{p+\nu-1}}\nonumber\\
& = & \dfrac{(p+\nu-1)^{\frac{1}{p+\nu-1}}p}{(p+\nu)^{\frac{p+\nu}{p+\nu-1}}}\left[\dfrac{(p-1)!}{H}\right]^{\frac{1}{p+\nu-1}}g^{\frac{p+\nu}{p+\nu-1}}.
\end{eqnarray*}
Note that
\begin{eqnarray*}
\dfrac{(p+\nu-1)^{\frac{1}{p+\nu-1}}p}{(p+\nu)^{\frac{p+\nu}{p+\nu-1}}}&=&\dfrac{p(p+\nu-1)^{\frac{1}{p+\nu-1}}}{(p+\nu)(p+\nu)^{\frac{p+\nu}{p+\nu-1}}}\geq\left(\dfrac{p}{p+1}\right)\left(\dfrac{p-1}{p+1}\right) \\
%
&\geq & \dfrac{1}{3}.
\end{eqnarray*}
Thus,
$
h(r^{*})\geq\dfrac{1}{3}\left[\dfrac{(p-1)!}{H}\right]^{\frac{1}{p+\nu-1}}g^{\frac{p+\nu}{p+\nu-1}}
$
and so by (\ref{eq:A.21}) we get (\ref{eq:A.20}).
\end{proof}

Finally, for the case in which $\nu$ is unknown, the next lemma establishes that (4.6) is also achievable when the regularization parameter is sufficiently large.

\begin{lemma}
\label{lem:A.7}
Let $H_{f,p}(\nu)<+\infty$ for some $\nu\in [0,1]$, and assume that $x^{+}$ satisfies
\begin{equation}
\|\nabla\Omega_{\bar{x},p,H}^{(1)}(x^{+})\|_{*}\leq\theta\|x^{+}-\bar{x}\|^{p}
\label{eq:A.23}
\end{equation}
for some $\bar{x}\in\E$, $H>0$, and $\theta\geq 0$. If for some $\delta>0$ we have
\begin{equation}
\|\nabla f(x^{+})\|_{*}\geq\delta\quad\text{and}\quad H\geq\max\left\{C\theta (p-1)!,(CH_{f,p}(\nu))^{\frac{p}{p+\nu-1}}\left(\dfrac{4}{\delta}\right)^{\frac{1-\nu}{1+\nu}}\right\}
\label{eq:A.24}
\end{equation}
with $C\geq 4$, then
\begin{equation}
\la\nabla f(x^{+}),\bar{x}-x^{+} \ra\geq\dfrac{1}{4}\left[\dfrac{(p-1)!}{H}\right]^{\frac{1}{p}}\|\nabla f(x^{+})\|_{*}^{\frac{p+1}{p}}.
\label{eq:A.25}
\end{equation}
\end{lemma}

\begin{proof}
Denote $r=\|x^{+}-\bar{x}\|$. Then, by (\ref{eq:2.5}), (\ref{eq:2.7}), and (\ref{eq:A.23}) we have
\begin{eqnarray*}
\|\nabla f(x^{+})+\dfrac{H(p+1)}{p!}r^{p-1}B(x^{+}-\bar{x})\|_{*} & = &\|\nabla f(x^{+})-\nabla\Phi_{\bar{x},p}(x^{+})+\nabla\Omega_{\bar{x},p,H}^{(1)}(x^{+})\|_{*}\\
&\leq &\|\nabla f(x^{+})-\nabla\Phi_{\bar{x},p}(x^{+})\|_{*}+\|\nabla\Omega_{\bar{x},p,H}(x^{+})\|_{*}\\
&\leq &\dfrac{H_{f,p}(\nu)}{(p-1)!}r^{p+\nu-1}+\theta r^{p}\\
&  \leq  &\left(\dfrac{H}{C(p-1)!}+\theta\right)r^{p}
\end{eqnarray*}
Therefore,
\begin{eqnarray*}
\left(\dfrac{H}{C(p-1)!}+\theta\right)^{2}r^{2p}&\geq & \|\nabla f(x^{+})+\dfrac{H}{p!}r^{p-1}B(x^{+}-\bar{x})\|_{*}^{2}\\
& = & \|\nabla f(x^{+})\|_{*}^{2}+\dfrac{2(p+1)}{p!}Hr^{p-1}\la f^{(1)}(x^{+}),x^{+}-\bar{x}\ra \dfrac{H^{2}(p+1)^{2}}{(p!)^{2}}r^{2p},
\end{eqnarray*}
which gives
\begin{equation}
\la \nabla f(x^{+}),\bar{x}-x^{+}\ra \geq\dfrac{p!}{2(p+1)Hr^{p-1}}\|\nabla f(x^{+})\|_{*}^{2}+\dfrac{p!}{2H(p+1)}\left[\left(\dfrac{H(p+1)}{p!}\right)^{2}-\left(\dfrac{H}{C(p-1)!}+\theta\right)^{2}\right]r^{p+1}.
\label{eq:A.26}
\end{equation}
Since $H\geq C\theta(p-1)!$, it follows that
\begin{equation*}
\dfrac{p!}{2H(p+1)}\left[\left(\dfrac{H(p+1)}{p!}\right)^{2}-\left(\dfrac{H}{C(p-1)!}+\theta\right)^{2}\right]r^{p+1}=\dfrac{Hp!}{2(p+1)}\left[\left(\dfrac{p+1}{p!}\right)^{2}-\left(\dfrac{2}{C(p-1)!}\right)^{2}\right]^{2}
\end{equation*}
Because $C\geq 4$, we have
\begin{equation*}
-\left(\dfrac{2}{C(p-1)!}\right)^{2}\geq -\left(\dfrac{1}{2(p-1)!}\right)^{2}
\end{equation*}
and so
\begin{eqnarray*}
\dfrac{p!}{2H(p+1)}\left[\left(\dfrac{H(p+1)}{p!}\right)^{2}-\left(\dfrac{H}{C(p-1)!}+\theta\right)^{2}\right]r^{p+1}&\geq &\dfrac{Hp!}{2(p+1)}\left[\left(\dfrac{p+1}{p}\right)^{2}-\left(\dfrac{1}{2(p-1)!}\right)^{2}\right]\\
& \geq & \dfrac{3Hp^{2}}{8(p+1)!}.
\end{eqnarray*}
Therefore,
\begin{equation}
\la \nabla f(x^{+}),\bar{x}-x^{+}\ra\geq\dfrac{p!}{2H(p+1)r^{p-1}}\|\nabla f(x^{+})\|^{2}+\dfrac{3Hp^{2}}{8(p+1)!}r^{p+1}.
\label{eq:A.27}
\end{equation}
Denote $g=\|\nabla f(x^{+})\|$ and consider the right-hand side of (\ref{eq:A.26}) as a function of $r$:
\begin{equation*}
h(r)=\dfrac{p!}{2(p+1)Hr^{p-1}}g^{2}+\dfrac{3Hp^{2}r^{p+1}}{8(p+1)!}.
\end{equation*}
Let us find the optimal $r_{*}$ as a solution to the first-order optimality condition for function $h$:
\begin{equation*}
\dfrac{(p-1)g^{2}p!}{2(p+1)Hr_{*}^{p}}=\dfrac{3Hp^{2}(p+1)r_{*}^{p}}{8(p+1)!}=\dfrac{3Hp^{2}r^{p}}{8p!}.
\end{equation*}
Solving this equation for $r_{*}$, we obtain
\begin{equation*}
r_{*}^{p}=\dfrac{g(p-1)!}{H\Delta}\sqrt{\dfrac{8(p-1)}{3(p+1)}}.
\end{equation*}
Consequently,
\begin{eqnarray*}
h(r_{*})& = & r_{*}\left[\dfrac{gp}{2(p+1)}\sqrt{\dfrac{6(p+1)}{8(p-1)}}+\dfrac{3gp}{8(p+1)}\sqrt{\dfrac{8(p-1)}{6(p+1)}}\right]\\
        & \geq &\dfrac{3gp}{8(p+1)}\left[\dfrac{6(p+1)+8(p-1)}{\sqrt{[8(p-1)][6(p+1)]}}\right]\left[\dfrac{g(p-1)!}{H}\sqrt{\dfrac{8(p-1)}{6(p+1)}}\right]^{\frac{1}{p}}\\
        &\geq &\dfrac{1}{4}\left[\dfrac{(p-1)!}{H}\right]^{\frac{1}{p}}g^{\frac{p+1}{p}}.
\end{eqnarray*}

Therefore, (\ref{eq:A.25}) holds.
\end{proof}

%

\section*{Acknowledgments}
The authors are very grateful to the two anonymous referees, whose comments helped to improve the first version of this paper.

\end{document}